\documentclass[11pt]{article}

\usepackage[margin=1in]{geometry}
\usepackage{amsmath,amssymb,amsthm,mathtools}
\usepackage{microtype}
\usepackage{graphicx}
\usepackage[round]{natbib}
\usepackage{xcolor}
\usepackage{thmtools}
\usepackage{thm-restate}
\usepackage[colorlinks=true, citecolor=blue, hypertexnames=false]{hyperref}

\newtheorem{theorem}{Theorem}[section]
\newtheorem{lemma}[theorem]{Lemma}

\newtheorem{corollary}[theorem]{Corollary}
\newtheorem{assumption}[theorem]{Assumption}
\theoremstyle{remark}
\newtheorem{remark}[theorem]{Remark}

\newcommand\widecheck[1]{%
    #1%
    \text{\llap{%
        \smash{\raisebox{\dimexpr 2\depth-0.5ex\relax}{%
            \scalebox{1}[-1]{$\widehat{\phantom{#1}}$}%
        }}%
    }\rule{0pt}{1.1em}}%
}

\title{Wasserstein Bounds for Euler and Randomized Midpoint Langevin Sampling with Noisy Gradients}
\author{Yifei Xiong and Vinayak Rao
\vspace{0.7em}\\
Department of Statistics, Purdue University}
\date{\today}

\begin{document}
\maketitle

\begin{abstract}
Stochastic-gradient Langevin methods are sampling algorithms that replace the gradient of the log probability density by a noisy, possibly biased evaluation. We study how such gradient-evaluation errors affect the Euler and randomized midpoint discretizations of Langevin dynamics by comparing each noisy chain with its exact-gradient counterpart. For continuously differentiable, strongly log-concave targets with globally Lipschitz gradients, we derive comparison bounds under a sub-Gaussian condition on the gradient noise.
The resulting 2-Wasserstein bounds hold uniformly over iterations, and are linear in both the step size and the $L^2$ norm of the evaluation bias, with a  square-root dependence on dimensionality.
To the best of our knowledge, no comparable 2-Wasserstein bounds have previously been established for randomized midpoint with noisy gradients.
For Euler, our results complement recent Wasserstein bounds by allowing biased evaluations, and by not requiring additional regularity assumptions on the evaluation errors.

\end{abstract}

\section{Introduction}\label{sec:intro}

Langevin Monte Carlo (LMC) provides a simple way to sample from a density $\pi(x)\propto\exp(-U(x))$ using evaluations of $\nabla U$. In many settings, gradient evaluations can dominate the computational cost. For instance, if $\pi$ is a Bayesian posterior distribution based on a large data set, each exact gradient may require a pass through all observations.  Stochastic gradient Langevin dynamics (SGLD) replaces this calculation by a gradient estimate based on a small random subset of the data \citep{welling2011bayesian}. This reduces the cost of each step, but introduces an additional source of error into an already discretized process. We refer the reader to \citet{nemeth2021stochastic} for a review of stochastic gradient Monte Carlo.

An approximate drift evaluation can contain both a systematic bias and a centered random error. Even when the evaluation is unbiased, its random fluctuations can change the distribution sampled at a fixed step size. \citet{vollmer2016exploration} characterize the dependence of SGLD's asymptotic bias on the step size and gradient variance, and study the bias and variance of finite-time estimators. \citet{brosse2018promises} show that this variance can substantially affect sampling accuracy as the data set grows. Variance reduction offers one way to limit this effect \citep{chatterji2018theory}. 
For Euler discretization, \citet{dalalyan2019user} give nonasymptotic Wasserstein bounds that separate the contributions of bias and centered fluctuations.

\citet{li2022sharp} use relative entropy to compare SGLD with the Langevin diffusion uniformly in time, obtaining first-order Wasserstein accuracy under their regularity and initialization assumptions.
{More recently, \citet{paulin2026theoretical} obtain first-order Wasserstein bounds between the invariant distribution of SGLD and the target distribution under fourth-moment and mean-square Lipschitz conditions on centered gradient errors. Their analysis compares SGLD with a stationary exact-gradient Euler chain and also yields finite-time bounds relative to the target distribution.}

{
Here, we directly compare the distributions of the noisy and exact-gradient chains of the same discretization at each iteration, starting from a common initial distribution that need not be stationary. 
}
This quantifies the additional error from approximate gradients separately from the discretization.
We assume a sub-Gaussian condition on the centered errors, without requiring additional regularity of the evaluation errors.
The resulting comparison bounds are uniform over iterations and also allow biased evaluations.

These comparison bounds can be combined with exact-gradient error bounds to control the distance to the target distribution.
\citet{dalalyan2026improved} sharpen Wasserstein bounds for LMC using average coordinate-wise smoothness.
For twice continuously differentiable, strongly convex potentials with Lipschitz gradients, \citet{pedrotti2026wasserstein} establish an order-$\alpha$ bound on the 2-Wasserstein distance between the invariant distribution of exact-gradient Euler and the target distribution, where $\alpha$ is the step size. 
This motivates seeking an order-$\alpha$ bound for the effect of centered gradient errors, with a separate term for systematic bias.

We also study the randomized midpoint (RM) method. Introduced for log-concave sampling by \citet{shen2019randomized}, RM uses two drift evaluations per step, one at the current state and one at a random intermediate state.
Its stationary bias and asymptotic normality have been studied by \citet{he2020ergodicity}. 
For twice differentiable, strongly convex potentials with Lipschitz gradients, \citet{yu2024langevin} establish a Wasserstein discretization bound of order $\alpha$. This raises the question of whether RM retains this accuracy when both drift evaluations in each step are approximate.
Extending the comparison to RM requires accounting for the second evaluation at a random state that depends on both the first evaluation and the Gaussian noise within the step.

We address this question for conditionally sub-Gaussian centered errors, uniformly over the query points. 
For the Euler method, we obtain an order-$(\alpha+\delta)$ comparison between the noisy and exact-gradient chains in 2-Wasserstein distance, where $\delta$ controls the evaluation bias in $L^2$.
For RM, we allow both evaluations to be biased and obtain a uniform comparison bound of the same order. Both comparisons hold for continuously differentiable, strongly convex potentials with globally Lipschitz gradients. The bounds depend explicitly on the dimension $p$ through a factor of $\sqrt p$.
Combining them with the exact-gradient bounds under their additional assumptions yields sampling accuracy of order $\alpha$ for both methods when the evaluations are unbiased or $\delta=O(\alpha)$, once the initialization error has decayed to this scale. A fixed finite-variance Euler example shows that second-moment control alone cannot guarantee this first-order comparison.

The rest of the paper is organized as follows. Section~\ref{sec:prelim} introduces the discretizations and the drift-error model. Section~\ref{sec:first-order} presents the Euler and RM bounds, with proofs in Appendix~\ref{sec:proofs}. Numerical experiments are presented in Section~\ref{sec:exp}, and Section~\ref{sec:conclusion} discusses the implications and limitations of the results.

\section{Background and preliminaries}\label{sec:prelim}

Let $\pi$ be a probability distribution on $\mathbb R^p$ with density $\pi(x)\propto\exp(-U(x))$. The associated Langevin diffusion is
\begin{equation*}
\mathrm dX_t=-\nabla U(X_t)\mathrm dt+\sqrt2\mathrm dB_t,
\end{equation*}
where $B_t$ is a standard $p$-dimensional Brownian motion. Under the assumptions below, this diffusion has invariant distribution $\pi$. Langevin Monte Carlo approximates it by a discrete-time chain. The convergence properties of the diffusion and its discretizations have long been studied \citep{roberts1996exponential}. Our focus is the additional error introduced by approximating the drift in a discrete update.

We write $\mathcal L(X)$ for the distribution of a random vector $X$, and $\|\cdot\|$ for the Euclidean norm. For probability measures $\mu$ and $\nu$ with finite second moments, their 2-Wasserstein distance is
\begin{equation*}
W_2(\mu,\nu)=\inf_{\mathcal L(X)=\mu, \mathcal L(Y)=\nu}
\big(\mathbb E\|X-Y\|^2\big)^{1/2},
\end{equation*}
where the infimum is over joint distributions of $(X,Y)$ with the prescribed marginals. Any such coupling gives an upper bound on $W_2$. All initial distributions considered below have finite second moments.

For the theoretical analysis below, we impose the following conditions on the potential $U$.

\begin{assumption}[Strong convexity]\label{ass:convex}
The function $U:\mathbb{R}^p\to\mathbb{R}$ is continuously differentiable, and there exists a constant $m>0$ such that
\begin{equation*}
U(x)-U(y)-\nabla U(y)^\top(x-y) \ge \frac{m}{2}\|x-y\|_2^2, \quad \forall x,y\in\mathbb{R}^p.
\end{equation*}
\end{assumption}

\begin{assumption}[Smoothness]\label{ass:smooth}
There exists a constant $M>0$ such that
\begin{equation*}
\|\nabla U(x)-\nabla U(y)\|_2 \le M\|x-y\|_2, \quad \forall x,y\in\mathbb{R}^p.
\end{equation*}
\end{assumption}

We write $\kappa=M/m$ for the condition number. 
These assumptions are standard in nonasymptotic analyses of Langevin Monte Carlo \citep{dalalyan2017further,durmus2019high}.

\subsection{Exact-drift discretizations}\label{sec:disc}

The Euler discretization of the Langevin diffusion is
\begin{equation}\label{eq:euler}
X_{k+1}^{\mathrm{EU}}=X_k^{\mathrm{EU}}-\alpha\nabla U(X_k^{\mathrm{EU}})+\sqrt{2\alpha}Z_{k+1},
\quad Z_{k+1}\sim\mathcal N_p(0,I_p),
\end{equation}
where $\alpha>0$ is the step size and $Z_{k+1}$ are independent standard Gaussian random vectors.

Classical Euler bounds under Assumptions~\ref{ass:convex} and~\ref{ass:smooth} have a discretization term of order $\sqrt\alpha$ \citep[Theorem~1]{dalalyan2019user}. The following recent result gives an order-$\alpha$ stationary bias bound for $C^2$ potentials.

\begin{lemma}[Stationary bias of Euler discretization~\citep{pedrotti2026wasserstein}]\label{lem:euler-stationary-bias}
Suppose Assumptions~\ref{ass:convex} and~\ref{ass:smooth} hold, $U\in C^2(\mathbb R^p)$, and $\alpha\le1/M$. Let $\widehat\pi_\alpha$ denote the invariant distribution of the Euler chain~\eqref{eq:euler} with step size $\alpha$. Then
\begin{equation*}
W_2(\widehat\pi_\alpha,\pi)\le6\alpha M\sqrt{\frac pm}.
\end{equation*}
\end{lemma}
Lemma~\ref{lem:euler-stationary-bias} controls the 2-Wasserstein distance between the invariant distribution of the Euler chain and the target $\pi$. 
To obtain a bound after finitely many steps, we also need to account for initialization.
Combining the lemma with the contraction of the Euler update gives the following finite-time bound, with proof in Appendix~\ref{app:proof-prelim}.

\begin{restatable}[Finite-time Wasserstein rate of Euler discretization]{corollary}{eulerfiniterate}
\label{cor:euler-rate}
Under the assumptions of Lemma~\ref{lem:euler-stationary-bias}, for any initial state $X_0^{\mathrm{EU}}$ with finite second moment and every $k\ge0$,
\begin{equation}\label{eq:euler-rate}
W_2(\mathcal{L}(X_k^{\mathrm{EU}}),\pi) \le (1-m\alpha)^k W_2(\mathcal{L}(X_0^{\mathrm{EU}}),\pi) + 12\alpha M\sqrt{\frac pm}.
\end{equation}
\end{restatable}

The first term in~\eqref{eq:euler-rate} controls the effect of initialization and decays with $k$, while the second bounds the discretization error and is of order $\alpha$ for fixed $m,M,p$.
We will later combine this estimate with our comparison between the noisy and exact Euler chains to bound the distance of the noisy chain to $\pi$.

The randomized midpoint (RM) method~\citep{shen2019randomized,yu2024langevin} uses an intermediate drift evaluation to approximate the drift integral more accurately. It selects a uniformly randomized time within each step.
The identity $\int_0^\alpha f(s) \mathrm ds=\mathbb{E}(\alpha f(\alpha\tau))$ for integrable $f$ and $\tau\sim\operatorname{Unif}(0,1)$ motivates evaluating the drift at one random time within the step.
The intermediate update constructs $Y_k^{\mathrm{RM}}$ by taking an Euler step of length $\alpha\tau_{k+1}$ from $X_k^{\mathrm{RM}}$. The terminal update advances from $X_k^{\mathrm{RM}}$ to $X_{k+1}^{\mathrm{RM}}$, using the drift at $Y_k^{\mathrm{RM}}$. These updates are given by
\begin{equation}\label{eq:rm-update}
\begin{aligned}
Y_k^{\mathrm{RM}}&=X_k^{\mathrm{RM}}-\alpha\tau_{k+1}\nabla U(X_k^{\mathrm{RM}})+\sqrt{2\alpha\tau_{k+1}}Z'_{k+1},\\
X_{k+1}^{\mathrm{RM}}&=X_k^{\mathrm{RM}}-\alpha\nabla U(Y_k^{\mathrm{RM}})+\sqrt{2\alpha}\left(\sqrt{\tau_{k+1}}Z'_{k+1}+\sqrt{1-\tau_{k+1}}Z''_{k+1}\right),
\end{aligned}
\end{equation}
where $\tau_{k+1}\sim\operatorname{Unif}(0,1)$ and $Z'_{k+1},Z''_{k+1}\sim\mathcal N(0,I_p)$ are mutually independent. The shared vector $Z'_{k+1}$ gives the midpoint and terminal increments the dependence of Brownian increments at times $\alpha\tau_{k+1}$ and $\alpha$.

The following lemma gives a Wasserstein bound for the RM update.

\begin{lemma}[Wasserstein rate of randomized midpoint~\citep{yu2024langevin}]\label{lem:rm-rate}
Suppose Assumptions~\ref{ass:convex} and~\ref{ass:smooth} hold and $U$ is twice differentiable. If $M\alpha+\sqrt\kappa(M\alpha)^{3/2}\le 1/4$, then for every $k\ge 0$,
\begin{equation}\label{eq:rm-rate}
W_2(\mathcal{L}(X_k^{\mathrm{RM}}),\pi) \le 1.11e^{-mk\alpha/2}W_2(\mathcal{L}(X_0^{\mathrm{RM}}),\pi) +(2.4\sqrt{\kappa M\alpha}+1.77)M\alpha\sqrt{\frac{p}{m}}.
\end{equation}
\end{lemma}
For fixed $m,M,p$, the discretization terms in~\eqref{eq:euler-rate} and~\eqref{eq:rm-rate} are both of order $\alpha$. Neither estimate includes error from approximating the drift. We next introduce a model for these evaluations, which will be applied once per Euler step and twice per RM step.

\subsection{Noisy drift evaluations}\label{sec:noisy_drift}

Suppose that an evaluation at $x$ returns a random approximation $\widehat{\nabla U}(x)$ to $\nabla U(x)$. We separate its error into a bias and a centered fluctuation:
\begin{equation}\label{eq:oracle-decomposition}
\widehat{\nabla U}(x)=\nabla U(x)+b(x)+\xi(x),
\end{equation}
where $b(x)=\mathbb E[\widehat{\nabla U}(x)]-\nabla U(x)$ and $\mathbb E[\xi(x)]=0$. Here $\xi(x)$ denotes the centered error in one evaluation at $x$.

A common example is the gradient estimate used in SGLD \citep{welling2011bayesian}. For a potential $U=U_0+\sum_{i=1}^n U_i$, one can sample an index $I$ uniformly from $\{1,\ldots,n\}$ and return $\nabla U_0(x)+n\nabla U_I(x)$. Averaging over $I$ recovers $\nabla U(x)$, so this estimate has $b(x)=0$.

In a sampling algorithm, the query point is itself random. %
We assume that, conditional on this query point $X$, the centered error has the distribution of $\xi(x)$ at $x=X$ and is independent of all variables generated before that evaluation. In particular, $\mathbb E[\xi(X)\mid X]=0$.

For Euler, replacing the drift in~\eqref{eq:euler} by~\eqref{eq:oracle-decomposition} gives the noisy chain (EU-N)
\begin{equation*}
X_{k+1}^{\mathrm{EU-N}}=X_k^{\mathrm{EU-N}}-\alpha(\nabla U(X_k^{\mathrm{EU-N}})+b_{k+1}+\xi_{k+1})+\sqrt{2\alpha}Z_{k+1},
\end{equation*}
where we write
\begin{equation*}
b_{k+1}=b(X_k^{\mathrm{EU-N}}),\quad \xi_{k+1}=\xi(X_k^{\mathrm{EU-N}}).
\end{equation*}
The subscript $k+1$ labels the evaluation used in the update from $k$ to $k+1$. To quantify its effect, suppose that, for every $k\ge0$,
\begin{equation}\label{eq:bounded-second-moment}
\mathbb E\|b_{k+1}\|^2\le\delta^2p,\quad
\mathbb E\|\xi_{k+1}\|^2\le\sigma^2p.
\end{equation}
The following theorem relates these two error levels to the sampling accuracy.
\begin{theorem}[Noisy Euler bound {\citep[Theorem~4]{dalalyan2019user}}]\label{thm:classical-noisy-euler}
Suppose Assumptions~\ref{ass:convex} and~\ref{ass:smooth} hold, the evaluations satisfy the model above and~\eqref{eq:bounded-second-moment}, and $\alpha\le2/(m+M)$. Then, for every $k\ge0$,
\begin{align}\label{eq:noisy-euler-rate}
W_2(\mathcal L(X_k^{\mathrm{EU-N}}),\pi)
\le\ &(1-m\alpha)^kW_2(\mathcal L(X_0^{\mathrm{EU-N}}),\pi)+1.65\frac Mm\sqrt{\alpha p}\notag\\
&+\frac{\delta\sqrt p}{m}+\frac{\sigma^2\sqrt{\alpha p}}{1.65M+\sigma\sqrt m}.
\end{align}
\end{theorem}
The $\delta$ term in~\eqref{eq:noisy-euler-rate} need not vanish as $\alpha$ decreases, whereas the term involving a fixed $\sigma>0$ is of order $\sqrt\alpha$. These are upper bounds under second-moment conditions. In the next section, we impose a sub-Gaussian condition and compare noisy and exact-drift chains directly, obtaining an order-$\alpha$ bound for the contribution of centered errors.

\section{Wasserstein error under sub-Gaussian drift noise}
\label{sec:first-order}

We compare each noisy chain with the corresponding chain using exact gradients, started from the same initial distribution. This separates the error due to drift approximation from the discretization error. The triangle inequality then combines the comparison with Corollary~\ref{cor:euler-rate} and Lemma~\ref{lem:rm-rate}. The following assumption controls the conditional distribution of the centered evaluation error.

\begin{assumption}[Sub-Gaussian centered drift errors]\label{ass:subgaussian}
There is a constant $K>0$ such that, for every evaluation in~\eqref{eq:oracle-decomposition} at a possibly random query point $X$,
\begin{equation*}
\mathbb E[\exp(u^\top\xi(X))\mid X]\le\exp\left(\frac{K^2}{2}\|u\|^2\right),\quad \forall u\in\mathbb R^p.
\end{equation*}
In particular, $\mathbb{E}[\xi(X)] = 0$, and $\mathbb E[\|\xi(X)\|^2\mid X]\le K^2p$.
\end{assumption}

This assumption thus implies the centered-error condition in~\eqref{eq:bounded-second-moment} with $\sigma=K$. The conditional independence in Section~\ref{sec:noisy_drift} also ensures that the sub-Gaussian bound and conditional centering remain valid when conditioning on all variables generated before an evaluation. We use this property when constructing couplings below.

\subsection{Error rate for Euler discretization with noisy drift}
\label{sec:euler-case}

For Euler evaluations, we first consider the centered fluctuation separately from the bias. Conditional on the current state $x$, the Gaussian increment $\sqrt{2\alpha}Z$ is replaced by $\sqrt{2\alpha}Z-\alpha\xi(x)$, while the bias contributes the deterministic shift $-\alpha b(x)$. The next lemma bounds the Wasserstein cost of the centered perturbation. The proof is deferred to Appendix~\ref{app:euler-case}.

\begin{restatable}[Gaussian coupling bound]
{lemma}{gaussiansmoothing}
\label{lem:gaussian-smoothing}
Under Assumption~\ref{ass:subgaussian}, fix $x\in\mathbb{R}^p$ and let $Z\sim\mathcal{N}(0,I_p)$ be independent of $\xi(x)$. If $\alpha K^2\le1$, then
\begin{equation}\label{eq:gaussian-smoothing}
W_2^2\left(\mathcal{L}(\sqrt{2\alpha}Z-\alpha\xi(x)),\mathcal{N}(0,2\alpha I_p)\right) \le \frac{2}{3}\alpha^3K^4p.
\end{equation}
\end{restatable}

\begin{remark}
For intuition, in one dimension, if $\xi$ itself were Gaussian (that is $\xi \sim N(0,K^2)$), then $(\sqrt{2\alpha}Z -\alpha \xi)  \sim N(0,2\alpha + \alpha^2K^2)$. 
For two centered one-dimensional Gaussians,
\begin{equation*}
W^2_2(N(0,\sigma_1^2),N(0,\sigma^2_2))=(\sigma_1-\sigma_2)^2.
\end{equation*}
In our setting this equals $(\sqrt{2\alpha+\alpha^2K^2}-\sqrt{2\alpha})^2$. 
For $\alpha K^2 \ll 1$, we then have %
$W^2_2= O(\alpha^{3}K^4)$.
Lemma~\ref{lem:gaussian-smoothing} gives the same order for any error satisfying Assumption~\ref{ass:subgaussian}. For fixed $K$ and $p$, the centered perturbation $-\alpha\xi(x)$ has $L^2$ norm of order at most $\alpha$, whereas its addition to the independent Gaussian increment changes the distribution by at most order $\alpha^{3/2}$ in $W_2$.
\end{remark}

\begin{remark}\label{remark:sync-coupling}
Under the synchronous Gaussian coupling (that uses the same Gaussian vector $Z$ in both increments), the difference between the noisy and exact increments is $-\alpha\xi(x)$. Its mean squared norm is bounded by $\alpha^2\mathbb E\|\xi(x)\|^2\le\alpha^2K^2p$. 
This loses one power of $\alpha$ compared with \eqref{eq:gaussian-smoothing} and, after accumulation over successive Euler steps,
yields a $W_2$ bound of order $\sqrt\alpha$ rather than order $\alpha$.
\end{remark}

We apply the coupling in Lemma~\ref{lem:gaussian-smoothing} conditionally on the current states. The increment difference has conditional mean zero, so its contribution enters the squared-distance recurrence at order $\alpha^3$. Combining this bound with contraction of the Euler drift yields the following uniform comparison.

\begin{restatable}[Drift error for Euler]
{theorem}{eulerfirstorder}
\label{thm:euler-first-order}
Suppose Assumptions~\ref{ass:convex}, \ref{ass:smooth}, and~\ref{ass:subgaussian} hold, the Euler biases satisfy $\mathbb E\|b_{k+1}\|^2\le\delta^2p$ for some $\delta\ge0$ and every $k\ge0$, and the initial states $X_0^{\mathrm{EU-N}}$ and $X_0^{\mathrm{EU}}$ have the same distribution. If $\alpha\le1/(m+M)$ and $\alpha K^2\le1$, then for every $k\ge0$,
\begin{equation}
W_2(\mathcal{L}(X_k^{\mathrm{EU-N}}),\mathcal{L}(X_k^{\mathrm{EU}})) \le \frac{\delta\sqrt p}{m}+\alpha K^2\sqrt{\frac{(m+M)p}{3mM}}.
\label{eq:euler-first-order}
\end{equation}
\end{restatable}
For unbiased evaluations, $\delta=0$ and the comparison is of order $\alpha$. More generally, the bound is of order $\alpha+\delta$ for fixed $m,M,p,K$.

Finite second moments for the noise alone do not suffice for this first-order comparison, even for a quadratic potential.

\begin{restatable}[Finite-variance counterexample]{proposition}{eulercounterexample}
\label{prop:euler-counterexample}
Let $U(x)=x^2/2$, so that $\pi=\mathcal N(0,1)$. At each noisy Euler step, use the unbiased drift evaluation $x+\xi$, with independent errors $\xi$ having density $p_{\xi}(z)=\frac54|z|^{-7/2}\mathbf 1_{\{|z|\ge1\}}$, which satisfy $\mathbb E\xi=0$ and $\mathbb E\xi^2=5$. Start both exact and noisy gradient chains from $\pi$. For every integer $n\ge64$ and $\alpha=1/(2n)$,
\begin{equation}\label{eq:euler-finite-variance-counterexample}
\begin{aligned}
W_2(\mathcal L(X_n^{\mathrm{EU-N}}),\mathcal L(X_n^{\mathrm{EU}}))&\ge\frac{\alpha^{7/8}}{32},\\
W_2(\mathcal L(X_n^{\mathrm{EU-N}}),\pi)&\ge\frac{\alpha^{7/8}}{32}.
\end{aligned}
\end{equation}
In contrast, $W_2(\mathcal L(X_n^{\mathrm{EU}}),\pi)\le\alpha/2$.
\end{restatable}

Thus, at the fixed time $n\alpha=1/2$, neither the distance between the two chains nor the total sampling error admits an order-$\alpha$ bound under a second-moment assumption alone. The proof is given in Appendix~\ref{app:euler-case}.

Theorem~\ref{thm:euler-first-order} bounds the distance between the noisy and exact Euler chains. Adding the distance from the exact chain to $\pi$, as bounded in Corollary~\ref{cor:euler-rate}, gives the overall sampling error.

\begin{corollary}[Overall error of noisy Euler]\label{cor:noisy-euler-overall}
Under the assumptions of Theorem~\ref{thm:euler-first-order}, suppose further that $U\in C^2(\mathbb R^p)$. Then, for every $k\ge0$,
\begin{align*}
W_2(\mathcal{L}(X_k^{\mathrm{EU-N}}),\pi)
\le{}&(1-m\alpha)^kW_2(\mathcal{L}(X_0^{\mathrm{EU-N}}),\pi)+12\alpha M\sqrt{\frac pm}\\
&+\frac{\delta\sqrt p}{m}+\alpha K^2\sqrt{\frac{(m+M)p}{3mM}}.
\end{align*}
\end{corollary}

\begin{proof}[Proof of Corollary~\ref{cor:noisy-euler-overall}]
The step-size condition gives $\alpha\le1/(m+M)\le1/M$. Apply the triangle inequality through $\mathcal L(X_k^{\mathrm{EU}})$. Bound the distance between the two chains by Theorem~\ref{thm:euler-first-order} and the distance from the exact chain to $\pi$ by Corollary~\ref{cor:euler-rate}.
\end{proof}

For unbiased evaluations, or when $\delta=O(\alpha)$, both the drift-approximation and discretization terms in Corollary~\ref{cor:noisy-euler-overall} are of order $\alpha$.
Randomized midpoint also admits an order-$\alpha$ discretization bound by Lemma~\ref{lem:rm-rate}. We next establish the corresponding noisy-exact comparison when both evaluations in each RM step are approximate.

\subsection{Error rate for randomized midpoint with noisy drift}\label{sec:rm-case}

The randomized-midpoint update~\eqref{eq:rm-update} evaluates the drift twice. At both stages, we use the drift approximation in~\eqref{eq:oracle-decomposition}, with centered errors satisfying Assumption~\ref{ass:subgaussian}.
Denote the noisy chain and its intermediate states by $X_k^{\mathrm{RM-N}}$ and $Y_k^{\mathrm{RM-N}}$. For the update from step $k$ to $k+1$, write
\begin{align*}
b_{k+1}^{(1)}&=b(X_k^{\mathrm{RM-N}}),\quad
\xi_{k+1}^{(1)}=\xi(X_k^{\mathrm{RM-N}}),\\
b_{k+1}^{(2)}&=b(Y_k^{\mathrm{RM-N}}),\quad
\xi_{k+1}^{(2)}=\xi(Y_k^{\mathrm{RM-N}}).
\end{align*}
Here the superscripts distinguish the two evaluations. The resulting noisy update is
\begin{equation}\label{eq:noisy-rm-update}
\begin{aligned}
Y_k^{\mathrm{RM-N}}&=X_k^{\mathrm{RM-N}}-\alpha\tau_{k+1}(\nabla U(X_k^{\mathrm{RM-N}})+b_{k+1}^{(1)}+\xi_{k+1}^{(1)})+\sqrt{2\alpha\tau_{k+1}}Z'_{k+1},\\
X_{k+1}^{\mathrm{RM-N}}&=X_k^{\mathrm{RM-N}}-\alpha(\nabla U(Y_k^{\mathrm{RM-N}})+b_{k+1}^{(2)}+\xi_{k+1}^{(2)})\\
&\quad+\sqrt{2\alpha}(\sqrt{\tau_{k+1}}Z'_{k+1}+\sqrt{1-\tau_{k+1}}Z''_{k+1}).
\end{aligned}
\end{equation}
Suppose that, for some $\delta\ge0$,
\begin{equation}\label{eq:stagewise-bias-bound}
\mathbb E\|b_{k+1}^{(j)}\|^2\le\delta^2p,\qquad j=1,2,\quad k\ge0.
\end{equation}

The error in the first drift evaluation perturbs the midpoint by a quantity proportional to $\alpha$. By smoothness, its effect on the terminal drift update gains another factor of $\alpha$. The main difficulty is the centered error $\xi_{k+1}^{(2)}$ in the second evaluation: it enters the terminal update directly and, through its query point $Y_k^{\mathrm{RM-N}}$, need not be independent of the Gaussian increment in that step. To apply Lemma~\ref{lem:gaussian-smoothing}, we pair this error with the independent Gaussian increment in the next step.

Write $Z_{k+1}=\sqrt{\tau_{k+1}}Z'_{k+1}+\sqrt{1-\tau_{k+1}}Z''_{k+1}$ for the standard Gaussian vector in the terminal update of~\eqref{eq:noisy-rm-update}. 
To remove the current second-stage centered error from the terminal update,
define
\begin{equation}\label{eq:bar-X-definition}
\bar X_{k+1}:=X_k^{\mathrm{RM-N}}-\alpha\nabla U(Y_k^{\mathrm{RM-N}})-\alpha b_{k+1}^{(2)}+\sqrt{2\alpha}Z_{k+1}=X_{k+1}^{\mathrm{RM-N}}+\alpha\xi_{k+1}^{(2)},
\end{equation}
with $\bar X_0=X_0^{\mathrm{RM-N}}$ and $\xi_0^{(2)}=0$. This change costs at most
\begin{equation*}
W_2(\mathcal L(X_k^{\mathrm{RM-N}}),\mathcal L(\bar X_k))\le\alpha(\mathbb E\|\xi_k^{(2)}\|^2)^{1/2}\le\alpha K\sqrt p.
\end{equation*}
Thus, an order-$(\alpha+\delta)$ comparison between $\bar X_k$ and $X_k^{\mathrm{RM}}$ yields the desired comparison for the noisy chain. For $k\ge1$, $\bar X_k$ is determined before $\xi_k^{(2)}$ is sampled, so this error remains centered and sub-Gaussian conditional on $\bar X_k$.

To express the update in terms of $\bar X_k$, substitute $X_k^{\mathrm{RM-N}}=\bar X_k-\alpha\xi_k^{(2)}$ and obtain
\begin{equation*}
\bar X_{k+1}=\bar X_k-\alpha\nabla U(Y_k^{\mathrm{RM-N}})-\alpha b_{k+1}^{(2)}+(-\alpha\xi_k^{(2)}+\sqrt{2\alpha}Z_{k+1}).
\end{equation*}
The noise term in parentheses pairs the centered error from the preceding update with the independent Gaussian increment from the current update. For $\alpha K^2\le1$, applying Lemma~\ref{lem:gaussian-smoothing} conditionally on $\bar X_k$ gives
\begin{equation*}
W_2^2\left(\mathcal L\left(-\alpha\xi_k^{(2)}+\sqrt{2\alpha}Z_{k+1}\mid\bar X_k\right),\mathcal N(0,2\alpha I_p)\right)\le\frac23\alpha^3K^4p\quad\text{a.s.}
\end{equation*}
Let $\widetilde Z_{k+1}$ denote the standard Gaussian vector used in the terminal update of an exact randomized-midpoint step from $\bar X_k$, so that $\mathcal L(\widetilde Z_{k+1}\mid\bar X_k)=\mathcal N(0,I_p)$. We couple $\sqrt{2\alpha}\widetilde Z_{k+1}$ with the noise term in parentheses so that their difference satisfies
\begin{equation}\label{eq:T1-coupling}
T_{1,k}:=-\alpha\xi_k^{(2)}+\sqrt{2\alpha}(Z_{k+1}-\widetilde Z_{k+1}),\quad\mathbb E[T_{1,k}\mid\bar X_k]=0,\quad\mathbb E\|T_{1,k}\|^2\le\frac23\alpha^3K^4p.
\end{equation}
The full conditional construction is given in Appendix~\ref{app:rm-case}.

To separate the perturbation of the update from the difference in starting states, we compare $\bar X_{k+1}$ with an exact randomized-midpoint step $\widecheck X_{k+1}$ starting from $\bar X_k$. This step must preserve the dependence between the midpoint and terminal Gaussian increments. Set
\begin{equation*}
Z'''_{k+1}=\sqrt{1-\tau_{k+1}}Z'_{k+1}-\sqrt{\tau_{k+1}}Z''_{k+1}.
\end{equation*}
This orthogonal transformation gives two independent standard Gaussian vectors $Z_{k+1}$ and $Z'''_{k+1}$, whose joint distribution does not depend on $\tau_{k+1}$, and the inverse relation is
\begin{equation*}
\sqrt{2\alpha\tau_{k+1}}Z'_{k+1}=\sqrt{2\alpha}\tau_{k+1}Z_{k+1}+\sqrt{2\alpha\tau_{k+1}(1-\tau_{k+1})}Z'''_{k+1}.
\end{equation*}
We choose the preceding coupling independently of $\tau_{k+1}$ and $Z'''_{k+1}$. Keeping these two variables and replacing $Z_{k+1}$ by $\widetilde Z_{k+1}$ therefore gives the exact midpoint and terminal updates
\begin{align}
\widecheck Y_k &= \bar X_k-\alpha\tau_{k+1}\nabla U(\bar X_k)+\sqrt{2\alpha}\tau_{k+1}\widetilde Z_{k+1}+\sqrt{2\alpha\tau_{k+1}(1-\tau_{k+1})}Z'''_{k+1},\label{eq:check-Y}\\
\widecheck X_{k+1} &= \bar X_k-\alpha\nabla U(\widecheck Y_k)+\sqrt{2\alpha}\widetilde Z_{k+1}.\notag
\end{align}
Subtracting this terminal update from~\eqref{eq:bar-X-definition} leaves $T_{1,k}$, together with the terms arising from the different midpoint drift evaluations and the terminal bias. The following lemma collects these latter two contributions in $T_{2,k}$ and bounds the difference between $\bar X_{k+1}$ and $\widecheck X_{k+1}$.

\begin{restatable}[One-step control of the randomized-midpoint perturbation]{lemma}{randomizedmidpointcontrol}\label{lem:randomized-midpoint-control}
Suppose Assumptions~\ref{ass:convex}, \ref{ass:smooth}, and~\ref{ass:subgaussian} and the bias bounds~\eqref{eq:stagewise-bias-bound} hold. Use the coupling in~\eqref{eq:T1-coupling} and the exact step~\eqref{eq:check-Y}, and set
\begin{equation*}
T_{2,k}=-\alpha(\nabla U(Y_k^{\mathrm{RM-N}})-\nabla U(\widecheck Y_k))-\alpha b_{k+1}^{(2)}.
\end{equation*}
Then
\begin{equation*}
\bar X_{k+1}-\widecheck X_{k+1}=T_{1,k}+T_{2,k}.
\end{equation*}
If $\alpha M\le1$ and $\alpha K^2\le1$, then
\begin{equation*}
(\mathbb E\|T_{2,k}\|^2)^{1/2}
\le\sqrt{\frac23}M\alpha^{5/2}K^2\sqrt p+\sqrt2M\alpha^2K\sqrt p+\alpha\delta\sqrt p(1+M\alpha).
\end{equation*}
\end{restatable}

Next, to propagate the one-step comparison between $\bar X_{k+1}$ and $X_{k+1}^{\mathrm{RM}}$, we insert $\widecheck X_{k+1}$ between the two states:
\begin{equation*}
\bar X_{k+1}-X_{k+1}^{\mathrm{RM}}
=(\bar X_{k+1}-\widecheck X_{k+1})+(\widecheck X_{k+1}-X_{k+1}^{\mathrm{RM}}).
\end{equation*}
Lemma~\ref{lem:randomized-midpoint-control} controls the first difference. The second term compares two exact-drift steps from $\bar X_k$ and $X_k^{\mathrm{RM}}$, using the same randomized time and Gaussian increments. The following contraction bound controls the effect of their different starting states.

\begin{restatable}[Exact randomized-midpoint contraction]{lemma}{exactrmcontraction}\label{lem:exact-rm-contraction}
Suppose Assumptions~\ref{ass:convex} and~\ref{ass:smooth} hold.
Let $(Y_k^{\mathrm{RM}},X_{k+1}^{\mathrm{RM}})$ and $(\widecheck Y_k,\widecheck X_{k+1})$ be exact randomized-midpoint steps from $X_k^{\mathrm{RM}}$ and $\bar X_k$, respectively. Couple them using the same randomized time and Gaussian increments in~\eqref{eq:rm-update}. If $\alpha\le1/(m+M)$, then
\begin{equation*}
\|\widecheck X_{k+1}-X_{k+1}^{\mathrm{RM}}\|^2 \le\left(1-\frac{2mM\alpha}{m+M}(1-M\alpha)^2+(m+M)M^2\alpha^3\right)\|\bar X_k-X_k^{\mathrm{RM}}\|^2.
\end{equation*}
If $\alpha\le1/(8M\sqrt\kappa)$, the factor on the right-hand side is at most $1-(47/64)m\alpha$.
\end{restatable}

We combine the perturbation bound in Lemma~\ref{lem:randomized-midpoint-control} with the contraction in Lemma~\ref{lem:exact-rm-contraction} to obtain a recurrence for $W_2(\mathcal L(\bar X_k),\mathcal L(X_k^{\mathrm{RM}}))$. 
Under the conditional coupling, the cross term between $T_{1,k}$ and $\bar X_k-X_k^{\mathrm{RM}}$ has zero expectation. The remaining cross term involves the difference between the two drift updates, which contains an additional factor of $\alpha$. We handle $T_{2,k}$ using its $L^2$ bound. 
Iteration gives a bound uniform in $k$. Since replacing $\bar X_k$ by $X_k^{\mathrm{RM-N}}$ costs at most $\alpha K\sqrt p$ in 2-Wasserstein distance, the triangle inequality yields the following theorem. The proof is given in Appendix~\ref{app:rm-case}.

\begin{restatable}[Drift error for randomized midpoint]{theorem}{randomizedmidpointbias}\label{thm:randomized-midpoint-bias}
Suppose Assumptions~\ref{ass:convex}, \ref{ass:smooth}, and~\ref{ass:subgaussian} hold, the biases satisfy~\eqref{eq:stagewise-bias-bound}, and the initial states $X_0^{\mathrm{RM-N}}$ and $X_0^{\mathrm{RM}}$ have the same distribution. If
\begin{equation*}
\alpha\le\min\left\{\frac1{8M\sqrt\kappa},\frac1{K^2}\right\},
\end{equation*}
then, for every $k\ge0$,
\begin{equation}\label{eq:rm-bias-bound}
W_2(\mathcal L(X_k^{\mathrm{RM-N}}),\mathcal L(X_k^{\mathrm{RM}}))\le\alpha\sqrt p\left(10\kappa K+\frac{K^2}{\sqrt m}\right)+\frac{4\delta\sqrt p}{m}.
\end{equation}
\end{restatable}

Adding the exact-chain error from Lemma~\ref{lem:rm-rate} to Theorem~\ref{thm:randomized-midpoint-bias} gives the following bound relative to the target distribution $\pi$.

\begin{theorem}[Overall error of noisy randomized midpoint]\label{thm:noisy-rm-overall}
Under the assumptions of Theorem~\ref{thm:randomized-midpoint-bias}, suppose further that $U$ is twice differentiable. Then, for every $k\ge0$,
\begin{align}
W_2(\mathcal L(X_k^{\mathrm{RM-N}}),\pi)&\le1.11e^{-mk\alpha/2}W_2(\mathcal L(X_0^{\mathrm{RM-N}}),\pi)+(2.4\sqrt{\kappa M\alpha}+1.77)M\alpha\sqrt{\frac pm}\notag\\
&\quad+\alpha\sqrt p\left(10\kappa K+\frac{K^2}{\sqrt m}\right)+\frac{4\delta\sqrt p}{m}.\label{eq:noisy-rm-overall}
\end{align}
\end{theorem}

\begin{proof}[Proof of Theorem~\ref{thm:noisy-rm-overall}]
The step-size condition implies
\begin{equation*}
M\alpha+\sqrt\kappa(M\alpha)^{3/2}
\le\frac{1}{8\sqrt\kappa}+\frac{1}{16\sqrt2\kappa^{1/4}}
\le\frac18+\frac{1}{16\sqrt2}
<\frac14,
\end{equation*}
so Lemma~\ref{lem:rm-rate} applies. The triangle inequality gives
\begin{equation*}
W_2(\mathcal{L}(X_k^{\mathrm{RM-N}}),\pi)
\le W_2(\mathcal{L}(X_k^{\mathrm{RM-N}}),\mathcal{L}(X_k^{\mathrm{RM}}))+W_2(\mathcal{L}(X_k^{\mathrm{RM}}),\pi).
\end{equation*}
Since the two chains have the same initial distribution, combining Theorem~\ref{thm:randomized-midpoint-bias} with Lemma~\ref{lem:rm-rate} proves \eqref{eq:noisy-rm-overall}.
\end{proof}

\begin{remark}
For fixed $m,M,p,K$, the bound in~\eqref{eq:noisy-rm-overall} is of order $\alpha+\delta$ once the initialization term has decayed to this scale. Thus randomized midpoint retains order-$\alpha$ accuracy with unbiased drift evaluations, and also with biased evaluations when $\delta=O(\alpha)$.
\end{remark}

\section{Numerical experiments}\label{sec:exp}

We compare Euler and randomized midpoint in three examples. The strongly convex example examines the effects of the noise distribution and evaluation bias. The bimodal example explores behavior beyond strong convexity, while Bayesian logistic regression considers errors arising from minibatch gradients. Computational details are given in Appendix~\ref{app:exp}.

\paragraph{A strongly convex target.}
Consider $\pi(x)\propto\exp[-U(x)]$ with
\begin{equation}\label{eq:exp-potential}
U(x)=\frac{x^2}{2}+s^2\log\cosh\left(\frac{x}{s}\right),\quad s=0.15.
\end{equation}
Here, $1\le U''(x)\le2$. Both methods start from $\pi$ and run to $T=8$. We consider independent gradient noise from $\mathcal{N}(0,1)$, the Gaussian mixture $0.9\mathcal{N}(0,1)+0.1\mathcal{N}(0,11)$ (MoG), and the standard Student $t_{2.05}$ distribution. We first use unbiased gradients, then consider varied biases for RM with Gaussian noise. 
The Gaussian and MoG noises satisfy the sub-Gaussian condition in Assumption~\ref{ass:subgaussian}, whereas the Student noise has finite second moment but is not sub-Gaussian.

For Gaussian and MoG noise, the $W_2$ distances between noisy and exact-gradient chains decay close to first order in $\alpha$ (Fig.~\ref{fig:exp1}-a). The decay is approximately half order for Student noise over the steps considered. Euler and RM give nearly identical curves in this comparison. RM has smaller distances to $\pi$ under Gaussian and MoG noise (Fig.~\ref{fig:exp1}-b), whereas the two methods perform similarly under Student noise.

With biases of $0$ and $\alpha/4$, the distances decay close to first order (Fig.~\ref{fig:exp1}-c). A bias of $\sqrt\alpha/10$ gives approximately half-order decay, while a fixed bias of $0.02$ produces an error floor.

\begin{figure}[t]
\centering
\includegraphics[width=\linewidth]{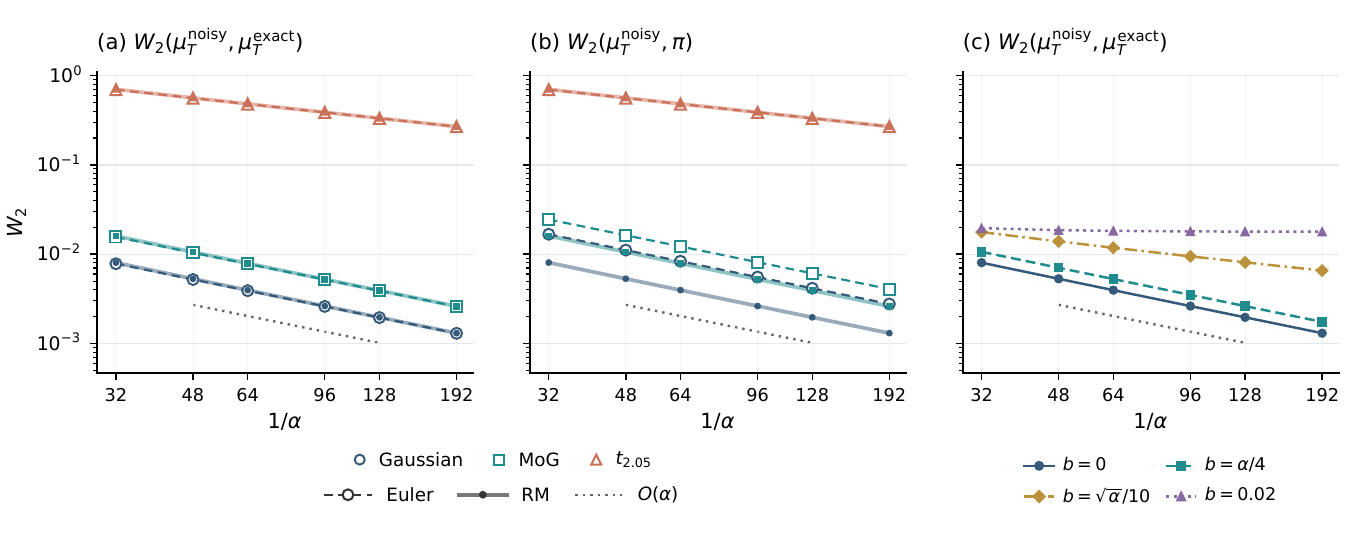}
\caption{Strongly convex example. $\mu_T^{\mathrm{noisy}}$ and $\mu_T^{\mathrm{exact}}$ denote terminal distributions with noisy and exact gradients. (a, b): Unbiased gradients under three noise distributions. (c): RM with Gaussian noise and varying gradient bias. Dotted lines indicate first-order decay.}
\label{fig:exp1}
\end{figure}

\paragraph{A bimodal target.}
We next consider
\begin{equation*}
\pi(x)\propto e^{-6(x+1)^2}+4e^{-18(x+1)^2}+e^{-8(x-1)^2},
\end{equation*}
a mixture of three Gaussian components with two modes. Its potential is nonconvex. We use the same three noise distributions, keep the gradients unbiased, and run both methods from $\pi$ to $T=32$.

For Gaussian and MoG noise, the distances between noisy and exact-gradient chains are similar for Euler and RM and decay close to first order in $\alpha$ (Fig.~\ref{fig:exp2}-a). The noisy RM distances to $\pi$ also decrease at roughly first order and remain below those of Euler (Fig.~\ref{fig:exp2}-b). Under Student noise, the distances between noisy and exact-gradient chains account for most of the error relative to $\pi$. With exact gradients, RM exhibits nearly second-order decay and has substantially smaller distances to $\pi$ than Euler (Fig.~\ref{fig:exp2}-c).

\begin{figure}[t]
\centering
\includegraphics[width=\linewidth]{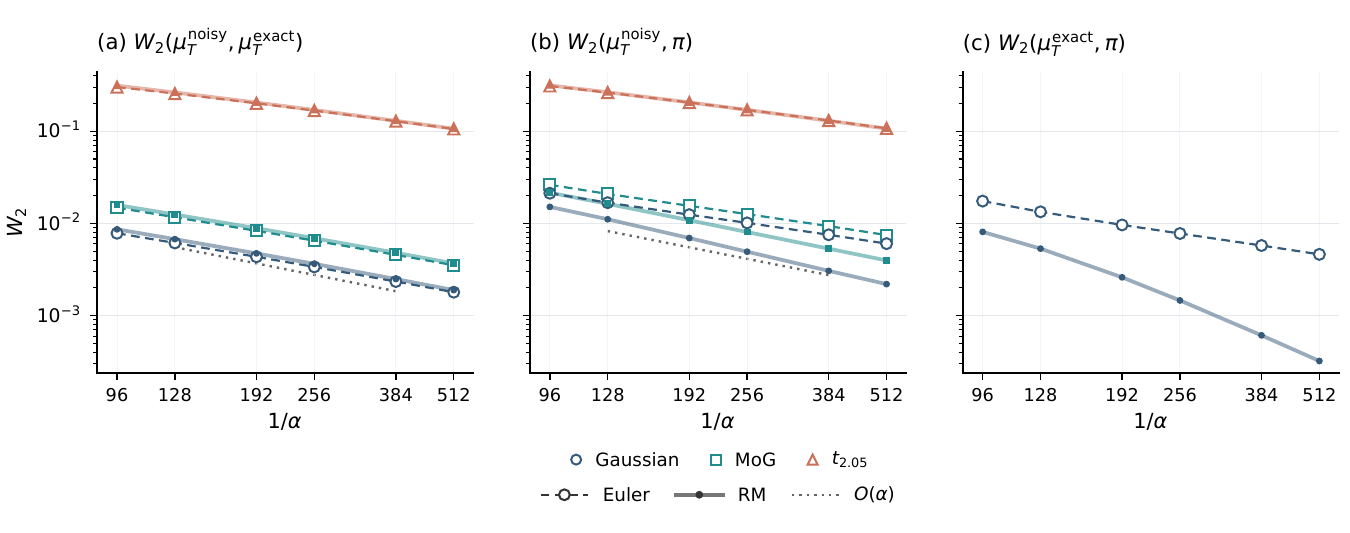}
\caption{Bimodal example. Distances between (a) noisy and exact-gradient chains, (b) noisy chains and $\pi$, and (c) exact-gradient chains and $\pi$. Dotted lines indicate first-order decay.}
\label{fig:exp2}
\end{figure}

\paragraph{Bayesian logistic regression.}
We next consider gradient errors arising from minibatch sampling. We fit a Bayesian logistic regression model to the $n=11{,}552$ MNIST training images labeled 3 or 5, a problem also considered by \citet{leimkuhler2024contraction}. The predictors are 50 whitened principal components and an intercept, with a standard Gaussian prior on the 51 coefficients. Writing $\ell_i(\theta)$ for the negative log-likelihood of observation $i$, we use the gradient estimate
\begin{equation}\label{eq:exp-minibatch}
g_B(\theta)=\theta+\frac nB\sum_{j=1}^B\nabla\ell_{I_j}(\theta), \quad I_j\overset{\mathrm{iid}}\sim\operatorname{Unif}\{1,\ldots,n\}.
\end{equation}
This estimate is unbiased, and its errors are uniformly bounded for the fixed data set, satisfying the conditional sub-Gaussian assumption. We use $B=1024$ or $4096$, with a fresh minibatch at every gradient evaluation. Both methods start at the posterior mode and run to $T=8$.

\begin{figure}[htbp]
\centering
\includegraphics[width=.95\linewidth]{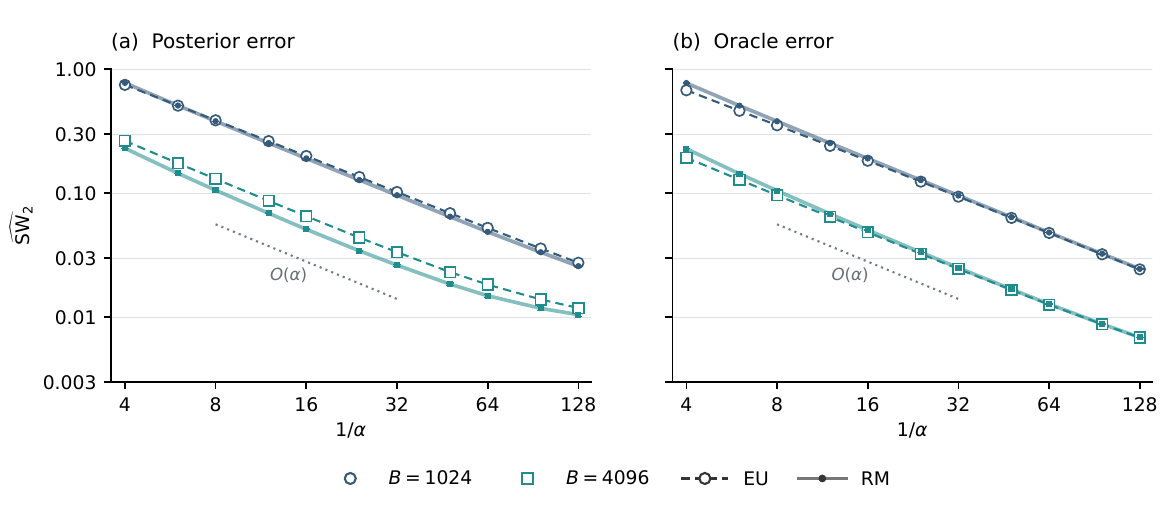}
\caption{Bayesian logistic regression. Sliced-$W_2$ distances to (a) the posterior reference and (b) the corresponding full-gradient chain. Points and bars show means and standard deviations over 20 repetitions. Dotted lines indicate first-order decay.}
\label{fig:exp3}
\end{figure}

We measure discrepancies by sliced-$W_2$ \citep{bonneel2015sliced}, the square root of the average squared $W_2$ distance over one-dimensional projections in Fig.~\ref{fig:exp3}. 
At $B=4096$, RM gives smaller distances to the posterior throughout the step-size range, although its distances to the full gradient chain are slightly larger than Euler's.
For both methods, distances to the corresponding full gradient chain decrease close to first order in $\alpha$, and increasing the batch size reduces these distances.

\section{Discussion}\label{sec:conclusion}
Our results quantify how noisy gradient evaluations affect Euler and randomized midpoint (RM) Langevin sampling. For strongly convex $U\in C^1$ with a globally Lipschitz gradient, our assumptions on the bias, together with the uniform sub-Gaussian noise yield an $O(\alpha+\delta)$ comparison in 2-Wasserstein distance with the corresponding exact-gradient chain, uniformly over iterations from a common initial distribution. This separates the effect of gradient errors from the discretization error, and both bounds carry an explicit $\sqrt p$ dimension factor.

Our main tool is a coupling of laws rather than paths: conditional centering and sub-Gaussianity make the noisy increment $\sqrt{2\alpha}Z-\alpha\xi$ close to a Gaussian at order $\alpha^{3/2}$ (Lemma~\ref{lem:gaussian-smoothing}), where pathwise coupling gives only order $\alpha$ and hence a square-root rate overall. In RM, the second evaluation is made at the random midpoint, so its error can depend on the Gaussian increment of the same step.
We handle this dependence by coupling the centered error with the Gaussian increment of the following step while preserving the joint distribution of the midpoint and terminal increments.
To the best of our knowledge, no comparable non-asymptotic 2-Wasserstein distance bounds have previously been established for randomized midpoint with noisy gradients.

For Euler, our comparison bounds complement recent results on the Wasserstein error of exact Euler discretization \citep{pedrotti2026wasserstein} and the sampling error of SGLD \citep{li2022sharp,paulin2026theoretical}, obtained under twice continuously differentiable potentials.
Under second moments alone, the noise contribution is known to be of order $\sqrt\alpha$ \citep{dalalyan2019user}. 
The SGLD analysis of \citet{paulin2026theoretical} combines fourth-moment control with mean-square Lipschitz regularity of unbiased errors to obtain an order $\alpha$ rate. 
Our sub-Gaussian condition imposes stronger tail control, but requires no continuity or Lipschitz regularity of the evaluation errors at the query point, and our bounds also allow biased evaluations.

Our counterexample uses a fixed noise distribution with finite variance and shows that second moments alone do not ensure a first-order comparison. A natural question is which weaker tail conditions retain both the first-order comparison rate and the square-root dependence on dimension. For RM, another question is how the noise level and the allocation of gradient samples between its two evaluations should depend on the step size to preserve gains from a more accurate discretization. The bimodal example from our empirical results shows that RM can have substantially smaller total error than Euler, although this setting lies outside the strong convexity assumptions of our analysis. Establishing conditions under which this accuracy advantage can be guaranteed for nonconvex targets remains an open question.

\section*{Acknowledgments}
The authors thank Prof.\ Arnak Dalalyan for bringing the recent work of Paulin, Pedrotti and Whalley on Langevin sampling to our attention.
VR acknowledges support from the National Science Foundation through grant DMS-2503118.

{
\small
\bibliographystyle{plainnat}
\bibliography{refs}
}

\appendix
\section{Proofs}\label{sec:proofs}

\subsection{Proofs for Section~\ref{sec:prelim}}\label{app:proof-prelim}

\eulerfiniterate*
\begin{proof}[Proof of Corollary~\ref{cor:euler-rate}]
Since $0\preceq I_p-\alpha\nabla^2U(x)\preceq(1-m\alpha)I_p$ for $\alpha\le1/M$,
the Euler drift map $x\mapsto x-\alpha\nabla U(x)$ is $(1-m\alpha)$-Lipschitz.
Coupling two chains with Euler update and the same Gaussian increments, with one chain initialized at its invariant distribution $\widehat\pi_\alpha$ and the other at $X_0^{\mathrm{EU}}$, therefore gives
\begin{equation*}
W_2(\mathcal L(X_k^{\mathrm{EU}}),\widehat\pi_\alpha) \le(1-m\alpha)^kW_2(\mathcal L(X_0^{\mathrm{EU}}),\widehat\pi_\alpha).
\end{equation*}
The triangle inequality and Lemma~\ref{lem:euler-stationary-bias} now yield
\begin{align*}
W_2(\mathcal L(X_k^{\mathrm{EU}}),\pi)
&\le(1-m\alpha)^kW_2(\mathcal L(X_0^{\mathrm{EU}}),\pi) +\left(1+(1-m\alpha)^k\right)W_2(\widehat\pi_\alpha,\pi)\\
&\le(1-m\alpha)^kW_2(\mathcal L(X_0^{\mathrm{EU}}),\pi) +12\alpha M\sqrt{\frac pm},
\end{align*}
which proves~\eqref{eq:euler-rate}.
\end{proof}

\subsection{Proofs for Section~\ref{sec:euler-case}}\label{app:euler-case}

\gaussiansmoothing*
\begin{proof}[Proof of Lemma~\ref{lem:gaussian-smoothing}]
Fix $x\in\mathbb{R}^p$ and write $\xi=\xi(x)$. Let $G\sim\mathcal{N}(0,I_p)$ be independent of $\xi$, set $t=\sqrt{\alpha/2}$, and write
\begin{equation*}
\mu_t=\mathcal{L}(G-t\xi), \quad \gamma=\mathcal{N}(0,I_p), \quad r_t=\frac{\mathrm d\mu_t}{\mathrm d\gamma}.
\end{equation*}
Talagrand's transportation inequality for the standard Gaussian measure~\citep{talagrand1996transportation} bounds the Wasserstein distance by relative entropy:
\begin{equation*}
W_2^2(\mu_t,\gamma)\le2D_{\mathrm{KL}}(\mu_t\|\gamma)=2\mathbb{E}_{Y\sim\mu_t}[\log r_t(Y)].
\end{equation*}
Using Jensen's inequality, and then $\mathrm d\mu_t=r_t\mathrm d\gamma$, we have
\begin{equation*}
\mathbb{E}_{Y\sim\mu_t}\log r_t(Y)\le\log\mathbb{E}_{Y\sim\mu_t}r_t(Y)=\log\int r_t(y)^2\gamma(\mathrm dy).
\end{equation*}
To bound this integral, note that for a fixed $a\in\mathbb{R}^p$, the density of $G+a$ relative to $G$ at $y$ is
\begin{equation*}
\exp\left(a^\top y-\frac12\|a\|^2\right).
\end{equation*}
Taking $a=-t\xi$ and averaging over $\xi$ gives
\begin{equation*}
r_t(y)=\mathbb{E}_{\xi} \left[\exp\left(-t\xi^\top y-\frac{t^2}{2}\|\xi\|^2\right)\right].
\end{equation*}
Let $\zeta$ be an independent copy of $\xi$. By Tonelli's theorem and the Gaussian moment-generating function,
\begin{align*}
\int r_t(y)^2\gamma(\mathrm dy)
&=\mathbb{E}_{\xi}\mathbb{E}_{\zeta}\left[\exp\left(-\frac{t^2}{2}\|\xi\|^2-\frac{t^2}{2}\|\zeta\|^2\right)
\mathbb{E}_{Y\sim\gamma}\exp\left(-t(\xi+\zeta)^\top Y\right)\right]\\
&=\mathbb{E}_{\xi}\mathbb{E}_{\zeta}\exp\left(-\frac{t^2}{2}\|\xi\|^2-\frac{t^2}{2}\|\zeta\|^2+\frac{t^2}{2}\|\xi+\zeta\|^2\right)\\
&=\mathbb{E}_{\xi}\mathbb{E}_{\zeta}\exp(t^2\zeta^\top\xi).
\end{align*}
We next apply Assumption~\ref{ass:subgaussian} twice. Conditional on $\zeta$, use the sub-Gaussian bound for $\xi$ with $u=t^2\zeta$ to obtain
\begin{equation*}
\mathbb{E}_{\xi} \left[\exp(t^2\zeta^\top\xi)\right]\le\exp\left(\frac{K^2t^4}{2}\|\zeta\|^2\right).
\end{equation*}
Let $\psi\sim\mathcal{N}(0,I_p)$ be independent of $\zeta$. For fixed $\zeta$, the Gaussian moment-generating function gives
\begin{equation*}
\exp\left(\frac{K^2t^4}{2}\|\zeta\|^2\right)=\mathbb{E}_{\psi}\exp(Kt^2\psi^\top\zeta).
\end{equation*}
Conditional on $\psi$, apply the sub-Gaussian bound for $\zeta$ with $u=Kt^2\psi$. It follows that
\begin{align*}
\mathbb{E}_{\zeta}\exp\left(\frac{K^2t^4}{2}\|\zeta\|^2\right)
&=\mathbb{E}_{\psi}\mathbb{E}_{\zeta}\exp(Kt^2\psi^\top\zeta)\\
&\le\mathbb{E}_{\psi}\exp\left(\frac{K^4t^4}{2}\|\psi\|^2\right)\\
&=\prod_{j=1}^p\mathbb{E}\exp\left(\frac{K^4t^4}{2}\psi_j^2\right)
=(1-K^4t^4)^{-p/2}.
\end{align*}
The last equality uses the moment generating function of a chi-squared random variable. It is finite because $K^4t^4\le1/4$ under the condition $\alpha K^2\le1$.
Combining the preceding estimates yields
\begin{equation*}
W_2^2(\mu_t,\gamma)\le-p\log(1-K^4t^4).
\end{equation*}
Since $-\log(1-z)\le z/(1-z)$ for $z\in[0,1)$,
\begin{equation*}
W_2^2(\mu_t,\gamma)\le\frac{pK^4t^4}{1-K^4t^4}.
\end{equation*}
Using $t^2=\alpha/2$ and $K^4t^4\le1/4$ gives
\begin{equation*}
W_2^2(\mu_t,\gamma)\le\frac13\alpha^2K^4p.
\end{equation*}
Finally, $\sqrt{2\alpha}(G-t\xi)=\sqrt{2\alpha}G-\alpha\xi$. By the homogeneity of $W_2$,
\begin{equation*}
W_2^2\left(\mathcal{L}(\sqrt{2\alpha}G-\alpha\xi),\mathcal{N}(0,2\alpha I_p)\right) =2\alpha W_2^2(\mu_t,\gamma) \le\frac23\alpha^3K^4p,
\end{equation*}
which proves \eqref{eq:gaussian-smoothing}.
\end{proof}

\eulerfirstorder*
\begin{proof}[Proof of Theorem~\ref{thm:euler-first-order}]
Let
\begin{equation*}
D_k=W_2(\mathcal{L}(X_k^{\mathrm{EU-N}}),\mathcal{L}(X_k^{\mathrm{EU}})).
\end{equation*}
For each $k$, choose an optimal coupling of the current states such that
\begin{equation*}
\mathbb{E}\|X_k^{\mathrm{EU-N}}-X_k^{\mathrm{EU}}\|^2=D_k^2.
\end{equation*}
Conditional on $(X_k^{\mathrm{EU-N}},X_k^{\mathrm{EU}})=(x,y)$, 
the random part of the noisy Euler update has distribution $\sqrt{2\alpha}Z_{k+1}-\alpha\xi(x)$, with $Z_{k+1}$ independent of the error. The corresponding random part of the exact Euler update is $\mathcal{N}(0,2\alpha I_p)$. By Lemma~\ref{lem:gaussian-smoothing}, we can couple these increments. Denote the coupled increments by $A_{k+1}$ and $\widetilde A_{k+1}$ so that
\begin{equation*}
A_{k+1}\sim\mathcal{L}(\sqrt{2\alpha}Z_{k+1}-\alpha\xi(x)), \quad \widetilde A_{k+1}\sim\mathcal{N}(0,2\alpha I_p),
\end{equation*}
and
\begin{equation*}
\mathbb{E}(\|A_{k+1}-\widetilde A_{k+1}\|^2\mid x,y)\le\frac23\alpha^3K^4p.
\end{equation*}
By Assumption~\ref{ass:subgaussian}, $\mathbb{E}\xi(x)=0$. Hence both increments are centered and 
\begin{equation*}
\mathbb{E}(A_{k+1}-\widetilde A_{k+1}\mid x,y)=0.
\end{equation*}

Under Assumptions~\ref{ass:convex} and~\ref{ass:smooth}, the strong convexity-smoothness interpolation inequality \citep[Theorem~2.1.12]{nesterov2013introductory} gives
\begin{equation}
\left\langle x-y,\nabla U(x)-\nabla U(y)\right\rangle
\ge\frac{mM}{m+M}\|x-y\|^2+\frac{1}{m+M}\|\nabla U(x)-\nabla U(y)\|^2.
\label{eq:interpolation}
\end{equation}
Therefore,
\begin{align*}
&\quad \ \|x-y-\alpha(\nabla U(x)-\nabla U(y))\|^2\\
&=\|x-y\|^2-2\alpha\left\langle x-y,\nabla U(x)-\nabla U(y)\right\rangle+\alpha^2\|\nabla U(x)-\nabla U(y)\|^2\\
&\le\left(1-\frac{2mM}{m+M}\alpha\right)\|x-y\|^2+\left(\alpha^2-\frac{2\alpha}{m+M}\right)\|\nabla U(x)-\nabla U(y)\|^2.
\end{align*}
Strong convexity and the Cauchy-Schwarz inequality imply
$\|\nabla U(x)-\nabla U(y)\|\ge m\|x-y\|$.
Since $\alpha\le1/(m+M)$, the coefficient of the last term above is nonpositive. Substituting this lower bound therefore gives
\begin{equation*}
\|x-y-\alpha(\nabla U(x)-\nabla U(y))\|^2
\le(1-m\alpha)^2\|x-y\|^2.
\end{equation*}

Use the coupled increments to define the next states by
\begin{align*}
X_{k+1}^{\mathrm{EU-N}}&=x-\alpha\nabla U(x)-\alpha b(x)+A_{k+1},\\
X_{k+1}^{\mathrm{EU}}&=y-\alpha\nabla U(y)+\widetilde A_{k+1}.
\end{align*}
The vector $x-y-\alpha(\nabla U(x)-\nabla U(y))-\alpha b(x)$ is fixed conditional on $(x,y)$. Its cross term with $A_{k+1}-\widetilde A_{k+1}$ therefore has conditional expectation zero. Expanding the squared norm gives
\begin{equation*}
\mathbb E(\|X_{k+1}^{\mathrm{EU-N}}-X_{k+1}^{\mathrm{EU}}\|^2\mid x,y)\le\|x-y-\alpha(\nabla U(x)-\nabla U(y))-\alpha b(x)\|^2+\frac23\alpha^3K^4p.
\end{equation*}
By the contraction bound, Minkowski's inequality, and $\mathbb E\|b_{k+1}\|^2\le\delta^2p$, we obtain
\begin{equation*}
D_{k+1}^2\le\big((1-m\alpha)D_k+\alpha\delta\sqrt p\big)^2+\frac23\alpha^3K^4p.
\end{equation*}
Set
\begin{equation*}
q=1-m\alpha,\quad B=\frac{\delta\sqrt p}{m},\quad S=\left(\frac{2\alpha^3K^4p}{3(1-q^2)}\right)^{1/2}.
\end{equation*}
Then $0\le q<1$, $\alpha\delta\sqrt p=(1-q)B$, and $D_0=0$. If $D_k\le B+S$, the recurrence implies
\begin{equation*}
D_{k+1}^2\le(B+qS)^2+(1-q^2)S^2 =B^2+2qBS+S^2\le(B+S)^2.
\end{equation*}
Induction yields $D_k\le B+S$ for every $k\ge0$. Finally,
\begin{equation*}
S^2=\frac{2\alpha^2K^4p}{3m(2-m\alpha)} \le\frac{\alpha^2K^4p(m+M)}{3mM},
\end{equation*}
where the last inequality follows from $\alpha\le1/(m+M)$. This proves~\eqref{eq:euler-first-order}.
\end{proof}

\eulercounterexample*
\begin{proof}[Proof of Proposition~\ref{prop:euler-counterexample}]
The stated density is symmetric and, for $t\ge1$, gives
\begin{equation*}
\mathbb E\xi=0,\quad \mathbb E\xi^2=\frac52\int_1^\infty z^{-3/2} \mathrm dz=5,\quad \Pr(|\xi|>t)=t^{-5/2}.
\end{equation*}
In particular, the errors have finite second moments but are not sub-Gaussian. Fix $n\ge64$ and $\alpha=1/(2n)$. Couple the chains using the same initial state $X_0^{\mathrm{EU-N}}=X_0^{\mathrm{EU}}\sim\pi$ and the same independent standard Gaussian $Z_k$. The updates are
\begin{align*}
X_{k+1}^{\mathrm{EU}}&=(1-\alpha)X_k^{\mathrm{EU}}+\sqrt{2\alpha}Z_{k+1},\\
X_{k+1}^{\mathrm{EU-N}}&=(1-\alpha)X_k^{\mathrm{EU-N}}+\sqrt{2\alpha}Z_{k+1}-\alpha\xi_{k+1}.
\end{align*}
Iterating the first recursion gives
\begin{equation*}
X_n^{\mathrm{EU}}=(1-\alpha)^nX_0^{\mathrm{EU}}+\sqrt{2\alpha}\sum_{j=1}^n(1-\alpha)^{n-j}Z_j.
\end{equation*}
The terms on the right are independent centered Gaussians, with total variance
\begin{align*}
(1-\alpha)^{2n}+2\alpha\sum_{j=0}^{n-1}(1-\alpha)^{2j}&=(1-\alpha)^{2n}+\frac{2\alpha(1-(1-\alpha)^{2n})}{1-(1-\alpha)^2}\\
&=1+\frac{\alpha(1-(1-\alpha)^{2n})}{2-\alpha}.
\end{align*}
Consequently, the Gaussian formula for $W_2$ gives
\begin{align*}
\mathcal L(X_n^{\mathrm{EU}})&=\mathcal N\left(0,1+\frac{\alpha(1-(1-\alpha)^{2n})}{2-\alpha}\right),\\
W_2(\mathcal L(X_n^{\mathrm{EU}}),\pi)&=\sqrt{1+\frac{\alpha(1-(1-\alpha)^{2n})}{2-\alpha}}-1\\
&\le\sqrt{1+\alpha}-1=\frac{\alpha}{\sqrt{1+\alpha}+1}\le\frac\alpha2,
\end{align*}
where $0<\alpha\le1$ implies $0\le1-(1-\alpha)^{2n}\le1$ and $2-\alpha\ge1$.

Subtracting the two updates gives
\begin{equation*}
X_{k+1}^{\mathrm{EU-N}}-X_{k+1}^{\mathrm{EU}}=(1-\alpha)(X_k^{\mathrm{EU-N}}-X_k^{\mathrm{EU}})-\alpha\xi_{k+1}.
\end{equation*}
Since the initial states are the same, iterating this recursion yields
\begin{equation}\label{eq:counterexample-euler-expansion}
X_n^{\mathrm{EU-N}}=X_n^{\mathrm{EU}}-\alpha\sum_{j=1}^n(1-\alpha)^{n-j}\xi_j.
\end{equation}
Here $X_n^{\mathrm{EU}}$ is independent of all the $\xi_j$. For $1\le j\le n$, we have $(1-\alpha)^{n-j}\ge1-(n-j)\alpha\ge1/2$. Thus $|\xi_j|>4\alpha^{-3/2}$ implies $\alpha(1-\alpha)^{n-j}|\xi_j|>2\alpha^{-1/2}$. We next show that, conditional on $X_n^{\mathrm{EU}}$ and $|\xi_1|,\ldots,|\xi_n|$, the probability of $|X_n^{\mathrm{EU-N}}|>2\alpha^{-1/2}$ is at least $1/2$ whenever at least one error exceeds this threshold.

For $u,v\in\mathbb R$ and $t\ge0$, the inequality $2|v|\le|u+v|+|u-v|$ implies that at least one of $|u+v|$ and $|u-v|$ exceeds $t$ whenever $|v|>t$. Hence
\begin{equation*}
\mathbf1_{\{|u+v|>t\}}+\mathbf1_{\{|u-v|>t\}}\ge\mathbf1_{\{|v|>t\}}.
\end{equation*}
Fix $j\in\{1,\ldots,n\}$ and write
\begin{equation*}
u=X_n^{\mathrm{EU}}-\alpha\sum_{i\ne j}^n(1-\alpha)^{n-i}\xi_i,\quad v=\alpha(1-\alpha)^{n-j}\xi_j.
\end{equation*}
Then $X_n^{\mathrm{EU-N}}=u-v$, and flipping the sign of $\xi_j$ changes this value to $u+v$. Conditional on $X_n^{\mathrm{EU}}$ and $|\xi_1|,\ldots,|\xi_n|$, the noise signs remain independent and equally likely to be positive or negative. Thus $(u,v)$ and $(u,-v)$ have the same conditional distribution. Applying the preceding indicator inequality with $t=2\alpha^{-1/2}$ gives
\begin{align*}
&\quad \Pr\left(|X_n^{\mathrm{EU-N}}|>2\alpha^{-1/2}\middle|X_n^{\mathrm{EU}},|\xi_1|,\ldots,|\xi_n|\right)\\
&=\frac12\mathbb E\left[\mathbf1_{\{|u-v|>2\alpha^{-1/2}\}}+\mathbf1_{\{|u+v|>2\alpha^{-1/2}\}}\middle|X_n^{\mathrm{EU}},|\xi_1|,\ldots,|\xi_n|\right]\\
&\ge\frac12\mathbf1_{\{\alpha(1-\alpha)^{n-j}|\xi_j|>2\alpha^{-1/2}\}}\ge\frac12\mathbf1_{\{|\xi_j|>4\alpha^{-3/2}\}}.
\end{align*}
The first inequality uses that $|v|=\alpha(1-\alpha)^{n-j}|\xi_j|$ is fixed under the conditioning. The left-hand side does not depend on $j$, so taking the maximum of the indicators on the right gives
\begin{align*}
&\quad \Pr\left(|X_n^{\mathrm{EU-N}}|>2\alpha^{-1/2}\middle|X_n^{\mathrm{EU}},|\xi_1|,\ldots,|\xi_n|\right)\\
&\ge\frac12\max_{1\le j\le n}\mathbf1_{\{|\xi_j|>4\alpha^{-3/2}\}}=\frac12\mathbf1_{\{\max_{1\le j\le n}|\xi_j|>4\alpha^{-3/2}\}}.
\end{align*}
Taking expectations now yields
\begin{equation}\label{eq:sign-arg}
\Pr(|X_n^{\mathrm{EU-N}}|>2\alpha^{-1/2})\ge\frac12\Pr\left(\max_{1\le j\le n}|\xi_j|>4\alpha^{-3/2}\right).
\end{equation}
The density of $\xi$ gives $\Pr(|\xi|>4\alpha^{-3/2})=(4\alpha^{-3/2})^{-5/2}=\alpha^{15/4}/32$. Since $\xi_1,\ldots,\xi_n$ are independent and $n\alpha=1/2$, we obtain
\begin{align*}
&\ \quad \Pr\left(\max_{1\le j\le n}|\xi_j|>4\alpha^{-3/2}\right)=1-\left(1-\frac{\alpha^{15/4}}{32}\right)^n\\
&\ge1-\exp\left(-\frac{n\alpha^{15/4}}{32}\right)=1-\exp\left(-\frac{\alpha^{11/4}}{64}\right)\ge\frac{\alpha^{11/4}}{128},
\end{align*}
where we used $1-x\le e^{-x}$ and $1-e^{-t}\ge t/2$ for $0\le t\le1$, noting that $\alpha^{11/4}/64\le1$.
Combining this bound with~\eqref{eq:sign-arg} gives
\begin{equation}\label{eq:counterexample-noisy-tail}
\Pr(|X_n^{\mathrm{EU-N}}|>2\alpha^{-1/2})\ge\frac{\alpha^{11/4}}{256}.
\end{equation}

Let $Y$ have distribution $\mathcal L(X_n^{\mathrm{EU}})$ or $\pi$, and couple it arbitrarily with $X_n^{\mathrm{EU-N}}$. In either case, $Y$ is centered Gaussian with variance at most $2$. On the event $\{|X_n^{\mathrm{EU-N}}|>2\alpha^{-1/2},|Y|\le\alpha^{-1/2}\}$, the triangle inequality gives $|X_n^{\mathrm{EU-N}}-Y|>\alpha^{-1/2}$. Hence, by~\eqref{eq:counterexample-noisy-tail} and the Gaussian tail bound,
\begin{align*}
\mathbb E|X_n^{\mathrm{EU-N}}-Y|^2&\ge\alpha^{-1}\Pr(|X_n^{\mathrm{EU-N}}|>2\alpha^{-1/2},|Y|\le\alpha^{-1/2})\\
&\ge\alpha^{-1}\left[\Pr(|X_n^{\mathrm{EU-N}}|>2\alpha^{-1/2})-\Pr(|Y|>\alpha^{-1/2})\right]\\
&\ge\alpha^{-1}\left[\frac{\alpha^{11/4}}{256}-2e^{-1/(4\alpha)}\right]\\
&\ge\frac{\alpha^{7/4}}{512}.
\end{align*}
Here $2e^{-1/(4\alpha)}\le\alpha^{11/4}/512$ for $0<\alpha\le1/128$. Taking the infimum over all such couplings and then square roots proves~\eqref{eq:euler-finite-variance-counterexample}.
\end{proof}

\subsection{Proofs for Section~\ref{sec:rm-case}}\label{app:rm-case}

\randomizedmidpointcontrol*
\begin{proof}[Proof of Lemma~\ref{lem:randomized-midpoint-control}]
Write $\tau=\tau_{k+1}$. The inverse Gaussian rotation gives
\begin{equation*}
\sqrt{2\alpha\tau}Z'_{k+1}
=\sqrt{2\alpha}\tau Z_{k+1}+\sqrt{2\alpha\tau(1-\tau)}Z'''_{k+1}.
\end{equation*}
Subtracting the midpoint equation in~\eqref{eq:check-Y} from that in~\eqref{eq:noisy-rm-update}, the terms containing $Z'''_{k+1}$ cancel. Since $X_k^{\mathrm{RM-N}}-\bar X_k=-\alpha\xi_k^{(2)}$, this gives
\begin{align*}
Y_k^{\mathrm{RM-N}}-\widecheck Y_k
&=-\alpha\xi_k^{(2)}-\alpha\tau(\nabla U(X_k^{\mathrm{RM-N}})-\nabla U(\bar X_k))\\
&\quad -\alpha\tau b_{k+1}^{(1)}-\alpha\tau\xi_{k+1}^{(1)} +\sqrt{2\alpha}\tau(Z_{k+1}-\widetilde Z_{k+1}).
\end{align*}
By the definition of $T_{1,k}$,
\begin{equation*}
-\alpha\xi_k^{(2)}+\sqrt{2\alpha}\tau(Z_{k+1}-\widetilde Z_{k+1})
=\tau T_{1,k}-\alpha(1-\tau)\xi_k^{(2)}.
\end{equation*}
Substituting this identity into the preceding difference yields
\begin{align}\label{eq:midpoint-perturbation}
Y_k^{\mathrm{RM-N}}-\widecheck Y_k
&=\tau T_{1,k}-\alpha\tau b_{k+1}^{(1)}\notag\\
&\quad -\alpha\left[(1-\tau)\xi_k^{(2)}+\tau(\nabla U(X_k^{\mathrm{RM-N}})-\nabla U(\bar X_k))+\tau\xi_{k+1}^{(1)}\right].
\end{align}
Using the $M$-smoothness of $U$ and then $M\alpha\le1$, the first two terms inside the brackets satisfy
\begin{equation*}
\|(1-\tau)\xi_k^{(2)}+\tau(\nabla U(X_k^{\mathrm{RM-N}})-\nabla U(\bar X_k))\|\le(1-\tau+M\alpha\tau)\|\xi_k^{(2)}\|\le\|\xi_k^{(2)}\|.
\end{equation*}
Next, the distribution of $\xi_{k+1}^{(1)}$ given $(\bar X_k,\xi_k^{(2)},\tau)$ depends only on its query point $X_k^{\mathrm{RM-N}}=\bar X_k-\alpha\xi_k^{(2)}$. Assumption~\ref{ass:subgaussian} therefore gives
\begin{equation*}
\mathbb E(\xi_{k+1}^{(1)}\mid\bar X_k,\xi_k^{(2)},\tau)=0,\quad
\mathbb E(\|\xi_{k+1}^{(1)}\|^2\mid\bar X_k,\xi_k^{(2)},\tau)\le K^2p.
\end{equation*}
Conditional on $(\bar X_k,\xi_k^{(2)},\tau)$, the first two terms inside the brackets are fixed, so their cross term with $\tau\xi_{k+1}^{(1)}$ has expectation zero. Consequently,
\begin{align*}
&\ \quad \mathbb E\|(1-\tau)\xi_k^{(2)}+\tau(\nabla U(X_k^{\mathrm{RM-N}})-\nabla U(\bar X_k))+\tau\xi_{k+1}^{(1)}\|^2\\
&=\mathbb E\|(1-\tau)\xi_k^{(2)}+\tau(\nabla U(X_k^{\mathrm{RM-N}})-\nabla U(\bar X_k))\|^2+\mathbb E[\tau^2\|\xi_{k+1}^{(1)}\|^2]\\
&\le\mathbb E\|\xi_k^{(2)}\|^2+K^2p\le2K^2p.
\end{align*}
Combining this estimate with~\eqref{eq:stagewise-bias-bound} and~\eqref{eq:T1-coupling}, and applying Minkowski's inequality to~\eqref{eq:midpoint-perturbation}, yields
\begin{align*}
(\mathbb E\|Y_k^{\mathrm{RM-N}}-\widecheck Y_k\|^2)^{1/2}
&\le(\mathbb E[\tau^2\|T_{1,k}\|^2])^{1/2}+\alpha(\mathbb E[\tau^2\|b_{k+1}^{(1)}\|^2])^{1/2}\\
&\quad+\alpha\bigl(\mathbb E\|(1-\tau)\xi_k^{(2)}+\tau(\nabla U(X_k^{\mathrm{RM-N}})-\nabla U(\bar X_k))+\tau\xi_{k+1}^{(1)}\|^2\bigr)^{1/2}\\
&\le(\mathbb E\|T_{1,k}\|^2)^{1/2}+\alpha(\mathbb E\|b_{k+1}^{(1)}\|^2)^{1/2}+\sqrt2\alpha K\sqrt p,\\
&\le\sqrt{\frac23}\alpha^{3/2}K^2\sqrt p+\sqrt2\alpha K\sqrt p+\alpha\delta\sqrt p.
\end{align*}
Finally, the definition of $T_{2,k}$, Assumption~\ref{ass:smooth}, and~\eqref{eq:stagewise-bias-bound} imply
\begin{align*}
(\mathbb E\|T_{2,k}\|^2)^{1/2}
&\le\alpha(\mathbb E\|\nabla U(Y_k^{\mathrm{RM-N}})-\nabla U(\widecheck Y_k)\|^2)^{1/2}+\alpha(\mathbb E\|b_{k+1}^{(2)}\|^2)^{1/2}\\
&\le M\alpha(\mathbb E\|Y_k^{\mathrm{RM-N}}-\widecheck Y_k\|^2)^{1/2}+\alpha\delta\sqrt p\\
&\le\sqrt{\frac23}M\alpha^{5/2}K^2\sqrt p+\sqrt2M\alpha^2K\sqrt p
+\alpha\delta\sqrt p(1+M\alpha).
\end{align*}
\end{proof}

\exactrmcontraction*
\begin{proof}[Proof of Lemma~\ref{lem:exact-rm-contraction}]
Use the coupling specified in the statement, with the same $\tau_{k+1}$, $\widetilde Z_{k+1}$, and $Z'''_{k+1}$ for the two exact steps.
Since the Gaussian terms cancel, the midpoint and terminal differences are
\begin{align*}
\widecheck Y_k-Y_k^{\mathrm{RM}}
&=\bar X_k-X_k^{\mathrm{RM}}-\alpha\tau_{k+1}\left(\nabla U(\bar X_k)-\nabla U(X_k^{\mathrm{RM}})\right),\\
\widecheck X_{k+1}-X_{k+1}^{\mathrm{RM}}
&=\bar X_k-X_k^{\mathrm{RM}}-\alpha\left(\nabla U(\widecheck Y_k)-\nabla U(Y_k^{\mathrm{RM}})\right).
\end{align*}
We have
\begin{equation}\label{eq:check-Y-lower-bound}
\|\widecheck Y_k-Y_k^{\mathrm{RM}}\|
\ge\|\bar X_k-X_k^{\mathrm{RM}}\| - \alpha\tau_{k+1}\|\nabla U(\bar X_k)-\nabla U(X_k^{\mathrm{RM}})\| \ge (1-M\alpha)\|\bar X_k-X_k^{\mathrm{RM}}\|.
\end{equation}
Next, expand the squared norm in the second identity and use the first identity to rewrite $\bar X_k-X_k^{\mathrm{RM}}$ in the cross term:
\begin{align*}
\|\widecheck X_{k+1}-X_{k+1}^{\mathrm{RM}}\|^2
&=\|\bar X_k-X_k^{\mathrm{RM}}\|^2+\alpha^2\|\nabla U(\widecheck Y_k)-\nabla U(Y_k^{\mathrm{RM}})\|^2\\
&\quad -2\alpha\left\langle\bar X_k-X_k^{\mathrm{RM}},\nabla U(\widecheck Y_k)-\nabla U(Y_k^{\mathrm{RM}})\right\rangle\\
&=\|\bar X_k-X_k^{\mathrm{RM}}\|^2+\alpha^2\|\nabla U(\widecheck Y_k)-\nabla U(Y_k^{\mathrm{RM}})\|^2 \\
&\quad -2\alpha\left\langle\widecheck Y_k-Y_k^{\mathrm{RM}},\nabla U(\widecheck Y_k)-\nabla U(Y_k^{\mathrm{RM}})\right\rangle\\
&\quad -2\alpha^2\tau_{k+1}\left\langle\nabla U(\bar X_k)-\nabla U(X_k^{\mathrm{RM}}),\nabla U(\widecheck Y_k)-\nabla U(Y_k^{\mathrm{RM}})\right\rangle.
\end{align*}
For the first inner product, applying the interpolation inequality~\eqref{eq:interpolation} yields
\begin{equation*}
\left\langle\widecheck Y_k-Y_k^{\mathrm{RM}},\nabla U(\widecheck Y_k)-\nabla U(Y_k^{\mathrm{RM}})\right\rangle
\ge\frac{mM}{m+M}\|\widecheck Y_k-Y_k^{\mathrm{RM}}\|^2+\frac{1}{m+M}\|\nabla U(\widecheck Y_k)-\nabla U(Y_k^{\mathrm{RM}})\|^2.
\end{equation*}
For the second inner product, Young's inequality with weight $m+M$ gives
\begin{align*}
&\quad -2\alpha^2\tau_{k+1}\left\langle \nabla U(\bar X_k)-\nabla U(X_k^{\mathrm{RM}}),\nabla U(\widecheck Y_k)-\nabla U(Y_k^{\mathrm{RM}})\right\rangle\\
&\le\frac{\alpha}{m+M}\|\nabla U(\widecheck Y_k)-\nabla U(Y_k^{\mathrm{RM}})\|^2+(m+M)\alpha^3\tau_{k+1}^2\|\nabla U(\bar X_k)-\nabla U(X_k^{\mathrm{RM}})\|^2.
\end{align*}
Combining these two bounds yields
\begin{align*}
\|\widecheck X_{k+1}-X_{k+1}^{\mathrm{RM}}\|^2
&\le\|\bar X_k-X_k^{\mathrm{RM}}\|^2-\frac{2mM\alpha}{m+M}\|\widecheck Y_k-Y_k^{\mathrm{RM}}\|^2\\
&\quad +\left(\alpha^2-\frac{\alpha}{m+M}\right)\|\nabla U(\widecheck Y_k)-\nabla U(Y_k^{\mathrm{RM}})\|^2\\
&\quad +(m+M)\alpha^3\tau_{k+1}^2\|\nabla U(\bar X_k)-\nabla U(X_k^{\mathrm{RM}})\|^2.
\end{align*}
The coefficient in the third term is nonpositive because $\alpha\le1/(m+M)$, so this term can be discarded. Assumption~\ref{ass:smooth} further gives
\begin{equation*}
\|\nabla U(\bar X_k)-\nabla U(X_k^{\mathrm{RM}})\| \le M\|\bar X_k-X_k^{\mathrm{RM}}\|.
\end{equation*}
Using this Lipschitz bound, the midpoint estimate~\eqref{eq:check-Y-lower-bound}, and $\tau_{k+1}\le1$ in the preceding squared-distance bound gives
\begin{equation*}
\|\widecheck X_{k+1}-X_{k+1}^{\mathrm{RM}}\|^2 \le \left[1-\frac{2mM\alpha}{m+M}(1-M\alpha)^2+(m+M)M^2\alpha^3\right]\|\bar X_k-X_k^{\mathrm{RM}}\|^2,
\end{equation*}
This proves the contraction bound in the lemma.

Moreover, for $u=M\alpha$,
\begin{equation*}
\frac{1}{m\alpha}\left(\frac{2mM\alpha}{m+M}(1-u)^2-(m+M)M^2\alpha^3\right)=\frac{2\kappa}{\kappa+1}(1-u)^2-(\kappa+1)u^2.
\end{equation*}
If $u\le1/(8\sqrt\kappa)$, then $u\le1/8$ because $\kappa\ge1$. Since
\begin{equation*}
\frac{2\kappa}{\kappa+1}(1-u)^2\ge(1-u)^2\ge\frac{49}{64},
\end{equation*}
whereas
\begin{equation*}
(\kappa+1)u^2\le\frac{\kappa+1}{64\kappa}\le\frac{1}{32}.
\end{equation*}
Consequently,
\begin{equation*}
\frac{2\kappa}{\kappa+1}(1-u)^2-(\kappa+1)u^2\ge\frac{49}{64}-\frac{1}{32}=\frac{47}{64},
\end{equation*}
which proves that the contraction factor is at most $1-(47/64)m\alpha$.
\end{proof}

\begin{lemma}[Scalar square-root recursion]
\label{lem:scalar-square-root-recursion}
Let $0<t_1\le1$ and $t_2,t_3,t_4\ge0$. If a nonnegative sequence satisfies
\begin{equation*}
s_{k+1}\le\sqrt{(1-t_1)s_k^2+2t_4s_k+t_2}+t_3,\quad s_0=0,
\end{equation*}
then, for every $k\ge0$,
\begin{equation*}
s_k\le\frac{2(t_3+t_4)}{t_1}+\sqrt{\frac{t_2}{t_1}}.
\end{equation*}
\end{lemma}

\begin{proof}[Proof of Lemma~\ref{lem:scalar-square-root-recursion}]
Set
\begin{equation*}
r=\frac{2(t_3+t_4)}{t_1}+\sqrt{\frac{t_2}{t_1}}.
\end{equation*}
Since $0<t_1\le1$, we have $r\ge t_3$. Direct calculation gives
\begin{equation*}
t_1r^2-2(t_3+t_4)r+t_3^2-t_2=t_3^2+2(t_3+t_4)\sqrt{\frac{t_2}{t_1}}\ge0.
\end{equation*}
Equivalently,
\begin{equation*}
(1-t_1)r^2+2t_4r+t_2\le r^2-2t_3r+t_3^2=(r-t_3)^2.
\end{equation*}
Since $r\ge t_3$, taking square roots gives
\begin{equation*}
\sqrt{(1-t_1)r^2+2t_4r+t_2}+t_3\le r.
\end{equation*}
The right side of the recurrence is nondecreasing in $s_k\ge0$. Thus $s_0=0$ and induction give $s_k\le r$ for every $k\ge0$.
\end{proof}

\randomizedmidpointbias*
\begin{proof}[Proof of Theorem~\ref{thm:randomized-midpoint-bias}]
Set
\begin{equation*}
D_k=W_2(\mathcal L(\bar X_k),\mathcal L(X_k^{\mathrm{RM}})),
\end{equation*}
and initialize the exact chain by $X_0^{\mathrm{RM}}=\bar X_0=X_0^{\mathrm{RM-N}}$. At step $k$, choose an optimal coupling such that
\begin{equation*}
\mathbb E\|\bar X_k-X_k^{\mathrm{RM}}\|^2=D_k^2.
\end{equation*}
For $k\ge1$, realize this coupling by sampling $X_k^{\mathrm{RM}}$ given $\bar X_k$, conditionally independently of the pair $(Y_{k-1}^{\mathrm{RM-N}},\xi_k^{(2)})$. Define the sigma-field
\begin{equation*}
\mathcal C_k=\sigma(\bar X_k,X_k^{\mathrm{RM}},Y_{k-1}^{\mathrm{RM-N}}).
\end{equation*}
By~\eqref{eq:bar-X-definition}, $\bar X_k$ and the query point $Y_{k-1}^{\mathrm{RM-N}}$ are determined before $\xi_k^{(2)}$ is drawn. Thus conditioning on $\mathcal C_k$ leaves $\xi_k^{(2)}$ centered and sub-Gaussian with parameter $K$.

Next, we draw $Z_{k+1}\sim\mathcal N(0,I_p)$ independently of $(\mathcal C_k,\xi_k^{(2)})$. Conditional on $\mathcal C_k$, apply Lemma~\ref{lem:gaussian-smoothing} to couple
\begin{equation*}
-\alpha\xi_k^{(2)}+\sqrt{2\alpha}Z_{k+1} \quad\text{and}\quad \sqrt{2\alpha}\widetilde Z_{k+1}.
\end{equation*}
Implement this coupling by keeping $(\xi_k^{(2)},Z_{k+1})$. Then $\mathcal L(\widetilde Z_{k+1}\mid\mathcal C_k)=\mathcal N(0,I_p)$, and the difference $T_{1,k}$ defined in~\eqref{eq:T1-coupling} satisfies
\begin{equation}\label{eq:conditional-T1-bound}
\mathbb E(T_{1,k}\mid\mathcal C_k)=0,\quad
\mathbb E(\|T_{1,k}\|^2\mid\mathcal C_k)\le\frac23\alpha^3K^4p.
\end{equation}

Afterwards, we draw mutually independent $\tau_{k+1}\sim\mathrm{Unif}[0,1]$ and $Z'''_{k+1}\sim\mathcal N(0,I_p)$, independently of $\mathcal C_k$ and $(\xi_k^{(2)},Z_{k+1},\widetilde Z_{k+1})$. Reconstruct $Z'_{k+1},Z''_{k+1}$ by the inverse rotation and apply~\eqref{eq:noisy-rm-update} starting from $X_k^{\mathrm{RM-N}}=\bar X_k-\alpha\xi_k^{(2)}$. For the two exact steps, use~\eqref{eq:check-Y} starting from $\bar X_k$ and $X_k^{\mathrm{RM}}$, with the same $\tau_{k+1},\widetilde Z_{k+1},Z'''_{k+1}$ resulting in $(\widecheck Y_k,\widecheck X_{k+1})$ and $(Y_k^{\mathrm{RM}},X_{k+1}^{\mathrm{RM}})$.

Lemma~\ref{lem:randomized-midpoint-control} gives
\begin{equation*}
\bar X_{k+1}-X_{k+1}^{\mathrm{RM}}
=\left(\widecheck X_{k+1}-X_{k+1}^{\mathrm{RM}}\right)+T_{1,k}+T_{2,k}.
\end{equation*}
We first control the exact-step difference $\widecheck X_{k+1}-X_{k+1}^{\mathrm{RM}}$ together with the error $T_{1,k}$. The step-size assumption implies $\alpha(m+M)\le1/4<1$, so Lemma~\ref{lem:exact-rm-contraction} yields
\begin{equation}\label{eq:exact-step-L2}
\mathbb E\|\widecheck X_{k+1}-X_{k+1}^{\mathrm{RM}}\|^2 \le\left(1-\frac{47m\alpha}{64}\right)D_k^2.
\end{equation}
Since the difference $\widecheck X_{k+1}-X_{k+1}^{\mathrm{RM}}$ and $T_{1,k}$ share $\widetilde Z_{k+1}$, their cross term need not vanish. We separate the current-state difference $\bar X_k-X_k^{\mathrm{RM}}$, whose cross term (with $T_{1,k}$) does vanish by conditional centering, from the drift correction introduced during the step. To bound that correction, first use~\eqref{eq:interpolation} to obtain
\begin{align*}
\|\widecheck Y_k-Y_k^{\mathrm{RM}}\|^2
&\le\|\bar X_k-X_k^{\mathrm{RM}}\|^2-\left(\frac{2\alpha\tau_{k+1}}{m+M}-\alpha^2\tau_{k+1}^2\right) \|\nabla U(\bar X_k)-\nabla U(X_k^{\mathrm{RM}})\|^2\\
&\le\|\bar X_k-X_k^{\mathrm{RM}}\|^2.
\end{align*}
The last inequality follows from $\alpha M\le1$ and $m\le M$. Consequently,
\begin{equation}\label{eq:cross-term-extra-alpha}
\|\widecheck X_{k+1}-X_{k+1}^{\mathrm{RM}}-(\bar X_k-X_k^{\mathrm{RM}})\|=\alpha\|\nabla U(\widecheck Y_k)-\nabla U(Y_k^{\mathrm{RM}})\| \le M\alpha\|\bar X_k-X_k^{\mathrm{RM}}\|.
\end{equation}
Since $\bar X_k-X_k^{\mathrm{RM}}$ is $\mathcal C_k$-measurable, \eqref{eq:conditional-T1-bound} gives
\begin{equation*}
\mathbb E\langle\bar X_k-X_k^{\mathrm{RM}},T_{1,k}\rangle=0.
\end{equation*}
Using~\eqref{eq:cross-term-extra-alpha} and the Cauchy-Schwarz inequality, we obtain
\begin{align*}
&\quad \ 2\mathbb E\langle\widecheck X_{k+1}-X_{k+1}^{\mathrm{RM}},T_{1,k}\rangle\\
&=2\mathbb E\langle\widecheck X_{k+1}-X_{k+1}^{\mathrm{RM}}-(\bar X_k-X_k^{\mathrm{RM}}),T_{1,k}\rangle\\
&\le2\left(\mathbb E\|\widecheck X_{k+1}-X_{k+1}^{\mathrm{RM}}-(\bar X_k-X_k^{\mathrm{RM}})\|^2\right)^{1/2} \left(\mathbb E\|T_{1,k}\|^2\right)^{1/2}\\
&\le2M\alpha D_k(\mathbb E\|T_{1,k}\|^2)^{1/2}.
\end{align*}

Next, we set
\begin{align*}
t_1&=\frac{47m\alpha}{64},\quad t_2=\frac23\alpha^3K^4p,\\
t_3&=\sqrt{\frac23}M\alpha^{5/2}K^2\sqrt p+\sqrt2M\alpha^2K\sqrt p+\alpha\delta\sqrt p(1+M\alpha).
\end{align*}
Combining this cross-term estimate with~\eqref{eq:conditional-T1-bound} and~\eqref{eq:exact-step-L2}, and applying Lemma~\ref{lem:randomized-midpoint-control} and Minkowski's inequality, gives
\begin{align*}
D_{k+1}&\le(\mathbb E\|\widecheck X_{k+1}-X_{k+1}^{\mathrm{RM}}+T_{1,k}\|^2)^{1/2}+(\mathbb E\|T_{2,k}\|^2)^{1/2}\\
&\le\left[(1-t_1)D_k^2+2M\alpha D_k(\mathbb E\|T_{1,k}\|^2)^{1/2}+\mathbb E\|T_{1,k}\|^2\right]^{1/2}+(\mathbb E\|T_{2,k}\|^2)^{1/2}\\
&\le\sqrt{(1-t_1)D_k^2+2M\alpha\sqrt{t_2}D_k+t_2}+t_3.
\end{align*}
The step size condition gives $0<t_1\le1$. Since $D_0=0$, applying Lemma~\ref{lem:scalar-square-root-recursion} with $t_4=M\alpha\sqrt{t_2}$ yields
\begin{equation*}
D_k\le\frac{2(t_3+M\alpha\sqrt{t_2})}{t_1}+\sqrt{\frac{t_2}{t_1}}.
\end{equation*}
Substituting the values of $t_1,t_2,t_3$ gives
\begin{align*}
D_k&\le\alpha\sqrt p\left[\frac{128\sqrt2}{47}\kappa K+\frac{256}{47}\sqrt{\frac23}\kappa\sqrt\alpha K^2+\sqrt{\frac{128}{141}}\frac{K^2}{\sqrt m}\right]+\frac{128\delta\sqrt p}{47m}(1+M\alpha).
\end{align*}
Since $\sqrt\alpha K^2\le K$, $M\alpha\le1/8$, and
\begin{equation*}
\frac{128}{47}\left(\sqrt2+2\sqrt{\frac23}\right)<9,\quad \sqrt{\frac{128}{141}}<1,\quad \frac{128}{47}(1+M\alpha)\le\frac{144}{47}<4,
\end{equation*}
we obtain
\begin{equation*}
D_k\le\alpha\sqrt p\left(9\kappa K+\frac{K^2}{\sqrt m}\right)+\frac{4\delta\sqrt p}{m}.
\end{equation*}
Finally, the defining relation between $X_k^{\mathrm{RM-N}}$ and $\bar X_k$ gives
\begin{equation*}
W_2(\mathcal L(X_k^{\mathrm{RM-N}}),\mathcal L(\bar X_k))\le\alpha(\mathbb E\|\xi_k^{(2)}\|^2)^{1/2}\le\alpha K\sqrt p.
\end{equation*}
The triangle inequality and $\kappa\ge1$ now prove~\eqref{eq:rm-bias-bound}.
\end{proof}

\section{Experimental details}\label{app:exp}

\subsection{One-dimensional experiments}\label{app:exp-settings}

Both Euler and RM methods start from $\pi$ and take $T/\alpha$ steps, with
\begin{equation*}
\begin{array}{c|c|l}
\text{Example}&T&1/\alpha\\ \hline
\text{Strongly convex}&8&32,\ 48,\ 64,\ 96,\ 128,\ 192\\
\text{Bimodal}&32&96,\ 128,\ 192,\ 256,\ 384,\ 512.
\end{array}
\end{equation*}
Gradient errors are independent across queries and independent of the Brownian increments and randomized times. In the bias comparison, the same constant is added at both RM evaluations.

For Gaussian and MoG errors, we compute the terminal densities by numerically evaluating the recursion
\begin{equation*}
p_{k+1}(y)=\int_{\mathbb R}p_k(x)P_\alpha(y\mid x) \mathrm dx,
\end{equation*}
where $P_\alpha(y\mid x)$ is the one-step transition density of Euler or RM. The integral over $x$ is evaluated on a uniform spatial grid. For RM, the transition density is obtained by integrating over the randomized time and Gaussian variables.

Integrating the terminal densities and inverting their CDFs gives the quantile functions. We compute their distances by using
\begin{equation}\label{eq:exp-quantile}
W_2^2(\mu,\nu)=\int_0^1\bigl[Q_\mu(u)-Q_\nu(u)\bigr]^2 \mathrm du.
\end{equation}
For Gaussian and MoG errors, we substitute $u=t^4/2$ and $u=1-t^4/2$ in the two halves of the integral.
For Student errors, we approximate central densities by scaled Hermite expansions and extend them with tail probabilities decaying as $|x|^{-2.05}$. In the strongly convex example, we evaluate the equivalent spatial integral with $x=\operatorname{sgn}(z)[(1-|z|)^{-20}-1]/20$, $-1<z<1$. In the bimodal example, we use $u=t^{41}/2$ and $u=1-t^{41}/2$ in the quantile integral.
The strongly convex example uses a grid on $[-8,8]$ with spacing $1/64$, and 12 time and 24 Gaussian nodes; the bimodal example uses $[-4,4]$ with spacing $1/256$ and 20 nodes for each RM integral. Grid densities are interpolated linearly. The Student expansions use degree 192 and scale $0.5$ in the strongly convex example, with degree 320 at $1/\alpha=32,96,192$, and degree 384 and scale $0.35$ in the bimodal example.

\subsection{Bayesian logistic regression}\label{app:exp-logistic}

Images are rescaled by $1/255$ before PCA. With labels $y_i\in\{-1,1\}$, the posterior potential is
\begin{equation*}
U(\theta)=\frac12\|\theta\|^2+\sum_{i=1}^n\ell_i(\theta),
\quad \ell_i(\theta)=\log\left(1+\exp({-y_i x_i^\top\theta})\right).
\end{equation*}
Since $\|\nabla\ell_i(\theta)\|\le\|x_i\|$, the errors in~\eqref{eq:exp-minibatch} are uniformly bounded, with conditional covariance proportional to $1/B$. Both methods run in coordinates $\theta=\widehat\theta+Lq$, where $\widehat\theta$ is the posterior mode and $LL^\top=[\nabla^2U(\widehat\theta)]^{-1}$. They use gradient $L^\top g_B(\widehat\theta+Lq)$, start from $q=0$, and take steps
\begin{equation*}
1/\alpha\in\{4,6,8,12,16,24,32,48,64,96,128\}.
\end{equation*}
These steps exceed the global sufficient bounds in the theorems. Euler uses one minibatch per step and RM uses two. Each configuration uses 20 independent repetitions, each with $N=65,536$ terminal samples.
The posterior reference uses an independence Metropolis sampler in $q$ coordinates with an identity-scale multivariate $t_{10}$ proposal. We run 64 independent chains, discard 2,000 iterations per chain, and retain 32,768 draws from each, giving $32N$ draws.

We measure discrepancies by sliced-$W_2$ distances \citep{bonneel2015sliced}, which use 256 fixed directions drawn uniformly from the unit sphere in $q$ coordinates. For direction $\ell$, let $\widehat Q_{\mu,\ell}$ and $\widehat Q_{\nu,\ell}$ be the empirical quantile functions of the projected samples. We compute
\begin{equation*}
\widehat{\mathrm{SW}}_2^2(\mu,\nu)=\frac1{256}\sum_{\ell=1}^{256}\int_0^1\left[\widehat Q_{\mu,\ell}(u)-\widehat Q_{\nu,\ell}(u)\right]^2 \mathrm{d}u.
\end{equation*}
For two sets of $N$ samples, each integral is the mean squared difference between corresponding sorted projections. For the posterior comparison, each noisy projection is paired with 32 consecutive sorted reference projections, and all $32N$ squared differences are averaged.

\end{document}